\documentclass[11pt,reqno]{amsart}
\usepackage{amsthm,amssymb,amsfonts}
\usepackage[nosumlimits]{mathtools}
\usepackage{mathrsfs}
\usepackage{graphicx}
\usepackage{accents}
\usepackage{algorithm}
\usepackage{algpseudocode}
\usepackage{tikz-cd}
\usepackage{enumerate}
\usepackage{quiver}
\usepackage{stmaryrd}
\usepackage{subcaption}
\usepackage{adjustbox}
\usepackage[margin=1in]{geometry}
\usepackage[shortlabels]{enumitem}

\usepackage{hyperref,xcolor}
\hypersetup{
    colorlinks,
    linkcolor={red!50!black},
    citecolor={blue!50!black},
    urlcolor={blue!80!black}
}

\usepackage{thmtools}
\usepackage{thm-restate}

\usepackage{tabularx}
\usepackage{array}

\newcommand{\tp}{{\scriptscriptstyle\mathsf{T}}}

\let\oldlhd\unlhd
\renewcommand*{\unlhd}{\mathrel{\mskip0.1mu \oldlhd \mskip0.1mu}}

\theoremstyle{plain}
\newtheorem{theorem}{Theorem}[section]
\newtheorem{proposition}[theorem]{Proposition}
\newtheorem{lemma}[theorem]{Lemma}
\newtheorem{corollary}[theorem]{Corollary}

\theoremstyle{definition}
\newtheorem{definition}[theorem]{Definition}

\newtheorem{question}[theorem]{Question}

\newtheorem{conjecture}[theorem]{Conjecture}

\DeclareMathOperator{\spa}{span}
\DeclareMathOperator{\Hom}{Hom}

\DeclareMathOperator{\Gr}{\mathsf{Gr}}

\DeclareMathOperator{\ch}{char}

\DeclareMathOperator{\ex}{ex}
\DeclareMathOperator{\zz}{z}
\DeclareMathOperator{\ttt}{t}
\DeclareMathOperator{\sss}{s}
\DeclareMathOperator{\rrr}{r}

\DeclareMathOperator{\ccc}{c}

\begin{document}
\hypersetup{hypertexnames=false}
\setlength{\emergencystretch}{1em}
\title{Random algebraic constructions for extremal and Ramsey problems}
\date{}
\author[Q-Y.~Chen]{Qiyuan~Chen}
\address{State Key Laboratory of Mathematical Sciences, Academy of Mathematics and Systems Science, Chinese Academy of Sciences, Beijing 100190, China}
\email{chenqiyuan@amss.ac.cn}
\author[H.~Liu]{Hong Liu}
\address{Extremal Combinatorics and Probability Group (ECOPRO), Institute for Basic Science (IBS), Daejeon, South Korea.}
\email{hongliu@ibs.re.kr}
\author[Y-A.~Yang]{Yuang~Yang}
\address{School of Mathematical Sciences, University of Science and Technology of China, Hefei 230026, China}
\email{Yangya2023@mail.ustc.edu.cn}
\author[K.~Ye]{Ke Ye}
\address{State Key Laboratory of Mathematical Sciences, Academy of Mathematics and Systems Science, Chinese Academy of Sciences, Beijing 100190, China}
\email{keyk@amss.ac.cn}
\begin{abstract}
Building on Bukh’s random algebraic method, we develop a framework for extremal and Ramsey
problems involving apex hypergraphs.  If $\mathcal H$ is a
$(d-1)$-partite $(d-1)$-uniform hypergraph with $S$ edges and
$\mathcal H(t)$ is obtained by adjoining $t$ vertices with common link
$\mathcal H$, we prove that
$\ex(n,\mathcal H(t))=\Omega_{\mathcal H}(n^{d-1/S})$
for $t>9^{S+o_d(S)}$, which is best possible when $\mathcal H$ is
Sidorenko.  Our
framework also yields sharper sided Zarankiewicz bounds, quantitative
generalized Tur\'an bounds, and diagonal multicolor Ramsey
constructions.  For each fixed $s\ge2$ and $K\ge3$, we further prove
$\rrr_K(\mathcal K_{s,t};\mathcal K_n)
=\Theta_{s,t,K}((n/\log n)^s)$ for $t>9^{s+o(s)}$,
extending a theorem of Alon and R\"odl from factorial to exponential
$t$.  The main ingredients are interpolation on $m$-independent
varieties, control of the dependencies imposed by symmetry, and
linear spaces of forms whose nonzero members remain regular after a
common algebraic slice.  Limited edge independence then gives the
spectral and local-density estimates needed for the Ramsey application.
\end{abstract}
\maketitle


\section{Introduction}

How dense can a hypergraph be if every prescribed configuration has only boundedly many common neighbors?  This question lies
behind several classical extremal and Ramsey problems.  It appears in the
Tur\'an and Zarankiewicz settings when the forbidden hypergraph is complete multipartite.  In generalized Tur\'an problems, the construction
must also maximize the number of copies of another hypergraph; in Ramsey
problems, one needs a decomposition into such hypergraphs and
additional control of their independent sets.  A construction that
serves all these purposes must therefore do substantially more than
avoiding a forbidden subhypergraph.

We study these questions through \emph{apex partite hypergraphs}.
Let $d\ge2$, let $\mathcal H$ be a $(d-1)$-partite
$(d-1)$-uniform hypergraph, and let $A$ be a set of $t$ new vertices.
Define $\mathcal H(t)$ by
\begin{equation}\label{eq:H(t)}
V(\mathcal H(t))=V(\mathcal H)\sqcup A,
\qquad
E(\mathcal H(t))=\{e\cup\{a\}:e\in E(\mathcal H),\ a\in A\}.
\end{equation}
Thus $\mathcal H(t)$ consists of $t$ apex vertices sharing the same
link $\mathcal H$.  Complete multipartite hypergraphs are special
cases, while arbitrary bases allow the common link to have a much
sparser and more varied structure.  The case $t=1$ is the usual
suspension~\cite{XDHJ26}.  Lee~\cite{lee2025sidorenko} recently
connected the extremal behavior of these hypergraphs to Sidorenko
exponents, bringing homomorphism inequalities into the problem of
constructing dense hypergraphs with bounded common links.

Our principal contribution is a multilinear extension of Bukh's
random algebraic method~\cite{Bukh15,Bukh_2024}, together with a
symmetric version that supports constructions on a single vertex
set, which allows a common mechanism for the different
requirements above, rather than separate constructions for each
application. The resulting bounds have two main features.  First, the relevant
exponent is governed by the number of edges of the base, not by the
size of its complete multipartite closure.  In particular, we verify
Lee's conjecture for every Sidorenko base.  Second, our bound extends the factorial-to-exponential
improvement for complete bipartite graphs to higher-uniformity Tur\'an
problems, and to diagonal and off-diagonal multicolor Ramsey problems.
Improved quantitative bounds are also obtained in the Zarankiewicz setting. We state these consequences below and explain the
new ingredients in Section~\ref{subsec:ideas}. Throughout, hypergraphs are finite and simple, and every forbidden
base has at least one edge.  All logarithms are natural.  The quantity $o_d(S)$ is understood as $S\to\infty$ with
$d$ fixed.

\subsection{Tur\'{a}n Numbers of Apex Partite Hypergraphs}\label{subsec:Turan}
For a $d$-uniform hypergraph $\mathcal H$, the \emph{Tur\'an number}
$\ex(n,\mathcal H)$ is the maximum number of edges in an $n$-vertex
$d$-uniform hypergraph containing no copy of $\mathcal H$.
Erd\H{o}s~\cite{Erdos64} proved that this number is
$O_{\mathcal H}(n^{d-\varepsilon})$ for some $\varepsilon>0$ if and
only if $\mathcal H$ is $d$-partite.  Determining the exponent within
this degenerate regime in general remains open.  For the complete
$d$-partite $d$-uniform hypergraph
$\mathcal K^{(d)}_{s_1,\dots,s_d}$, with
$s_1\le\cdots\le s_d$, the classical bound is
$\ex(n,\mathcal K^{(d)}_{s_1,\dots,s_d})
=O_{d,s_1,\dots,s_d}(n^{d-1/(s_1\cdots s_{d-1})})$
\cite{KST54,Erdos64}, which is conjectured to be sharp.

\begin{conjecture}\cite{Mubayi02}\label{conj:Turan}
For any positive integers $s_1 \le \cdots \le s_d$ with $s_d\ge2$,
\begin{equation}\label{conj:Turan:eq1}
\ex(n,\mathcal{K}^{(d)}_{
s_1,\dots,s_d}) = \Theta_{d,s_1,\dots, s_{d}}(n^{d - \frac{1}{s_1\cdots s_{d-1}}}).
\end{equation}
\end{conjecture}

The sharp exponent is known when the last part is sufficiently
large.  Ma, Yuan and Zhang~\cite{MYZ18} obtained it with an unspecified
threshold, and Pohoata and Zakharov~\cite{pohoata2021norm} gave the
factorial condition
$s_d>((d-1)(s_1\cdots s_{d-1}-1))!$.
For graphs, Bukh~\cite{Bukh_2024} reduced the threshold to an
exponential function of $s_1$.  Our first application extends this
exponential regime to every uniformity and, more generally, replaces
the complete multipartite base by an arbitrary partite hypergraph. Before stating our result, let us first motivate the study of such hypergraphs.

\subsubsection*{Connection to Sidorenko exponents}
For $r$-uniform hypergraphs $\mathcal H$ and $\mathcal G$, write
\[
\ttt_{\mathcal H}(\mathcal G)
=\frac{|\Hom(\mathcal H,\mathcal G)|}{|V(\mathcal G)|^{|V(\mathcal H)|}} \qquad \text{and}
\qquad
p(\mathcal G)=\ttt_{\mathcal K^{(r)}_r}(\mathcal G)
\]
for the homomorphism density of $\mathcal H$ in $\mathcal G$ and the edge homomorphism density,
respectively.  The \emph{Sidorenko exponent} of an $r$-partite
$r$-uniform hypergraph $\mathcal H$ is
\[
\sss(\mathcal H)=
\sup_{\mathcal G:\,0<p(\mathcal G)<1}
\frac{\log\ttt_{\mathcal H}(\mathcal G)}{\log p(\mathcal G)}.
\]
Equivalently, it is the least exponent $a$ such that
$\ttt_{\mathcal H}(\mathcal G)\ge p(\mathcal G)^a$ for every
$\mathcal G$. A hypergraph is \emph{Sidorenko} if $\sss(\mathcal H)=|E(\mathcal H)|$.
Sidorenko's conjecture asserts that every bipartite graph is
Sidorenko. In uniformity at least three, the analogous assertion is known to be
false~\cite{NS23,CJS24}.

Lee~\cite{lee2025sidorenko} proved that
$\ex(n,\mathcal H(t))=O_{\mathcal H,t}
(n^{d-1/\sss(\mathcal H)})$ for every $(d-1)$-partite base
$\mathcal H$.  This predicts a link between a homomorphism inequality
and the sharp exponent of an extremal construction.

\begin{conjecture}\cite{lee2025sidorenko}\label{conj:Lee}
Let $d \ge 2$ and $\mathcal{H}$ be a $(d-1)$-partite $(d-1)$-uniform hypergraph. There
exists a constant $C= C(\mathcal{H})$ such that for all $t \ge C$, we have $\ex(n,\mathcal{H}(t)) = \Theta_{\mathcal{H},t} \big(n^{d - 1/\sss(\mathcal{H})} \big)$.
\end{conjecture}

Lee singled out the case $\mathcal H=\mathcal C_6$, asking whether
$\ex(n,\mathcal C_6(t))=\Theta_t(n^{17/6})$ for all sufficiently
large $t$.  The following theorem answers this question and proves
the conjecture for every Sidorenko base.

\begin{theorem}[Tur\'{a}n Number]\label{thm-construction of apex partite}
Let $\mathcal H$ be a $(d-1)$--partite $(d-1)$--uniform hypergraph with $S$ edges. Denote $r \coloneqq \lfloor (6dS^2)^{1/3} \rfloor$. If $t >
3^{2S+r+3}(dS^2)^{1+\log r}r!$, then 
$$\ex(n,\mathcal{H}(t)) = \Omega_{\mathcal{H}}(n^{d-1/S}).$$
Consequently, Conjecture~\ref{conj:Lee} is true for every Sidorenko $\mathcal H$.
\end{theorem}

For fixed $d$, the threshold in Theorem~\ref{thm-construction of apex partite}
is $9^{S+o_d(S)}$, since $r=O_d(S^{2/3})$ and
$\log(r!)=O_d(S^{2/3}\log S)$. Here are some concrete families of examples.
\begin{itemize}
\item For the complete $d$--partite $d$--uniform hypergraph $\mathcal{K}^{(d)}_{
s_1,\dots,s_d}$, if $s_d> 9^{(1 + o_d(1))s_1\cdots s_{d-1}}$, then
$\ex(n,\mathcal{K}^{(d)}_{
s_1,\dots,s_d}) = \Theta_{d,s_1,\dots, s_{d}} \big(n^{d - \frac{1}{s_1\cdots s_{d-1}}} \big).$
\item For the even cycle graph $\mathcal{C}_{2\ell}$, if $t> 9^{(1 + o(1))2\ell}$, then 
$\ex \big(n,\mathcal{C}_{2\ell}(t)\big)
=\Theta_{\ell,t} \big(n^{3-\frac1{2\ell}}\big).$
\item For the hypercube $\mathcal{Q}_r$, if $t>
9^{(1 + o(1))r 2^{r-1}}$, then 
$\ex \big(n, \mathcal{Q}_r(t)\big) =\Theta_{r,t} \big(n^{3-\frac{1}{r2^{r-1}}}\big).$
\end{itemize}

\subsection{Sided Zarankiewicz Problem}
There is a useful distinction between forbidding all embeddings and
forbidding only those that respect specified parts.  In the latter
setting, one can make the base-side evaluations completely
independent without imposing geometric restrictions on those parts.
This leads both to a better quantitative bound and to constructions with
highly unequal part sizes.

Let $\mathcal{G}$ be a $d$--partite $d$--uniform hypergraph with ordered vertex sets $V_1,\dots,V_d$. Given a positive integer $t$ and a $(d-1)$--partite
$(d-1)$--uniform hypergraph $\mathcal{H}$ with ordered vertex parts, we say that $\mathcal{G}$ contains an \emph{ordered} $\mathcal{H}(t)$ if the $i$-part of $\mathcal{H}(t)$ is embedded in $V_i$ for every $i\in[d]$. 
The corresponding Zarankiewicz number is
\[
\zz(n_1,\dots,n_d; \mathcal{H}(t)) \coloneqq \max \{ |E(\mathcal{G})|: |V_i| = n_i,\; i\in[d],\; \mathcal{G} \text{~contains no ordered~} \mathcal{H}(t) \}.
\]
If $\mathcal{H} = \mathcal{K}^{(d-1)}_{s_1,\dots, s_{d-1}}$, then we write $\zz(n_1,\dots,n_d; s_1,\dots, s_{d})$ for $\zz(n_1,\dots,n_d; \mathcal{H}(s_d))$.

In the graph case, Conlon~\cite{ConlonZarankiewicz} proved $\zz(m,n;s,t)=\Omega_s(mn^{1-1/s})
$ when $\log_n m = O_s(t^{1/(s-1)})$, and raised the open question that follows.
\begin{question}\cite{ConlonZarankiewicz}\label{question:Zarankiewicz}
Is it true that for any $s$ and $t$ such that $2 \le s \le t$, we have $\zz(m,n;s,t) = \Omega_{s} (mn^{1-1/s})$ provided that $\log_n m \le t/s$. 
\end{question}
For $s\ge3$, Bukh~\cite[Theorem~2]{Bukh_2024} established the bound 
$\log_n m = O_s(t^{1/(1+2\log s)})$. He also obtained the optimal exponent in the balanced case for $t>3^{s+o(s)}$. For hypergraphs, Mubayi~\cite{MubayiPartite2026} established the balanced lower bound
\[
\zz(n,\dots,n;s_1,\dots,s_d) = n^{d-1/S-o(1)},
\]
for $s_d>3^{S+o(S)}$, where $S \coloneqq s_1 \cdots s_{d-1}$. In the unbalanced
setting, Gao, Hou, Huang and Wang~\cite{GaoHouHuangWang2025} proved that 
\[
\zz(n_1,\dots,n_d;s_1,\dots,s_d)
=\Omega_{s_1,\dots, s_d} \Big( n_d^{1-1/S} \prod_{i=1}^{d-1}n_i \Big)
\]
for $S\le s_d$ and $\log_{n_d} (\prod_{i=1}^{d-1}n_i)
\le s_d^{1/(S-1)}/(S(S-1))$. This lower bound is optimal if $n_d \le \min_{i\in[d-1]} n_i$. As an application of our framework, we obtain the following.

\begin{theorem}[Zarankiewicz number]
\label{thm:sided-Zarankiewicz}
Let $\mathcal H$ be a $(d-1)$--partite
$(d-1)$--uniform hypergraph with ordered vertex sets of sizes
$s_1,\dots,s_{d-1}$, and let $S\coloneqq |E(\mathcal H)|$. For each
positive integer $u$, define
\begin{equation}\label{thm:sided-Zarankiewicz:threshold}
\kappa_S(u) \coloneqq \min \Big\{m^S (2Su+1)^{\,1+\log\rho} \rho!: m,\rho \in \mathbb{N},\; \binom{m+\rho+1}{m}> 2Su + 1 \Big\}.
\end{equation}
If $t> \kappa_S(u)$, then, for all sufficiently large $n_d$ satisfying $\log_{n_d}\!(\prod_{i=1}^{d-1}n_i^{s_i}) \le u$, we have
\[
\zz(n_1,\dots,n_d;\mathcal{H}(t))
=\Omega_{\mathcal H,u} \Big( n_d^{1-1/S} \prod_{i=1}^{d-1}n_i \Big).
\]
\end{theorem}

The dependence on $u$ in \eqref{thm:sided-Zarankiewicz:threshold}
quantifies how much imbalance the construction can accommodate.
Two choices of the algebraic parameters give the following
consequences.

\begin{corollary}
\label{cor:sided-Zarankiewicz-complete}
Let $s_1,\dots,s_d$ be positive integers. Suppose that $S\coloneqq\prod_{i=1}^{d-1}s_i \ge2$.
\begin{enumerate}[\normalfont(i)]
\item\label{cor:sided-Zarankiewicz-complete:item1} If $s_d > 3^{S+o_d(S)}$, then $\zz(n,\dots,n;s_1,\dots,s_d) = \Theta_{d,s_1,\dots,s_d}
\left(n^{d-1/S}\right)$.
\item\label{cor:sided-Zarankiewicz-complete:item2} Denote $p_S \coloneqq S/(S+1) + 1 + \log S$. There exists a constant
$C \coloneqq C(s_1,\dots,s_{d-1})$ such that, for every sufficiently large $s_d$ and every sufficiently large $n_d$ depending on $s_1,\ldots,s_d$, if $\log_{n_d} \prod_{i=1}^{d-1}n_i \le C s_d^{1/p_S}$, 
then
\[
\zz(n_1,\dots,n_d;s_1,\dots,s_{d-1},s_d)
= \Omega_{d,s_1,\dots, s_{d}} \Big( n_d^{1-1/S} \prod_{i=1}^{d-1}n_i \Big).
\]
If moreover, it holds that $n_d \le \min_{i\in[d-1]} n_i$, then 
\[
\zz(n_1,\dots,n_d;s_1,\dots,s_{d-1},s_d)
= \Theta_{d,s_1,\dots, s_d}\Big( n_d^{1-1/S} \prod_{i=1}^{d-1}n_i \Big).
\]
\end{enumerate}
\end{corollary}

Part~\ref{cor:sided-Zarankiewicz-complete:item1} removes the
$o(1)$ loss in the exponent of $n$ from Mubayi's
bound~\cite{MubayiPartite2026}, without increasing its exponential
threshold.  For graphs,
part~\ref{cor:sided-Zarankiewicz-complete:item2} gives
$\zz(m,n;s,t)=\Omega_s(mn^{1-1/s})$ when
$\log_n m\le C_s t^{1/(2+\log s-1/(s+1))}$.
For each $s\ge3$, this improves the power of $t$ in Bukh's
range.  For $s\ge5$, it also improves on Conlon's range, but it
does not reach the linear range asked for in
Question~\ref{question:Zarankiewicz}.
Likewise, for $S\ge5$, it improves the power
$1/(S-1)$ in the range of Gao, Hou, Huang and
Wang~\cite{GaoHouHuangWang2025}.
Their generalized K\H{o}v\'ari--S\'os--Tur\'an bound gives the matching
upper estimate when $n_d\le\min_{i<d}n_i$.

\subsection{Generalized Tur\'{a}n Problem}
Let $\mathcal G$, $\mathcal F$, and $\mathcal H$ be $d$--uniform
hypergraphs. We denote by $\ccc(\mathcal{F},\mathcal{G})$ the number of
copies of $\mathcal{F}$ in $\mathcal{G}$. The \emph{generalized Tur\'{a}n
number} is defined by
\[
\ex(n,\mathcal F,\mathcal H)
\coloneqq
\max\bigl\{
\ccc(\mathcal F,\mathcal G):
|V(\mathcal G)|=n,\ 
\mathcal H\not\subseteq\mathcal G
\bigr\}.
\]

Ma, Yuan and Zhang~\cite{MYZ18} used random
algebraic constructions to obtain lower bounds when the forbidden hypergraph is complete multipartite
and its last part is sufficiently large.
Our symmetric framework extends the forbidden base to an arbitrary
partite hypergraph and gives an explicit apex threshold, exponential
in the combined number of base vertices and edges, under an
edge-count condition below, which is an artifact of our method.

\begin{theorem}[Generalized Tur\'{a}n Number]\label{thm:generalized Turan hypergraph explicit}
Let $\mathcal{B}$ be a $(d-1)$--partite $(d-1)$--uniform hypergraph,
and let $\mathcal{F}$ be a $d$--uniform hypergraph. Suppose that $|V(\mathcal B)| = \sigma$, $|E(\mathcal B)| = S$, $|V(\mathcal F)| = v$ and $|E(\mathcal F)| = e$. Denote $\rho\coloneqq \lfloor(7S\sigma)^{1/3} \rfloor$ and $T\coloneqq (S+\rho)\sigma + e +2$. If $e < (S\sigma)/6 -(7S\sigma^4)^{1/3}
-\min\{S-1,\sigma-d+1\}-2$, then for any $t>3^{2S+\rho+\sigma+d+3}T^{1+\log\rho}\rho!$, we have 
\[
\ex(n,\mathcal{F},\mathcal{B}(t))
=\Omega_{\mathcal{F},\mathcal{B}}(n^{v- \frac{e}{S}}).
\]
\end{theorem}

The exponent $v-e/S$ is the one suggested by a random hypergraph
of edge density $n^{-1/S}$.  The point is to retain this many copies
while excluding $\mathcal B(t)$ deterministically for every nonzero
form in a chosen linear space.  For complete multipartite counted
and forbidden hypergraphs, the upper bound of
Ma, Yuan and Zhang~\cite[Theorem~1.1]{MYZ18} gives a matching estimate.

\begin{corollary}\label{cor:generalized Turan-1}
Let $a_1,\dots,a_d,s_1,\dots,s_{d-1}$ be positive integers with
$a_1<s_1\le\cdots\le s_{d-1}$ and $a_i\le s_{i-1}$ for
$2\le i\le d$.  Put
$A=\prod_{i=1}^d a_i$, $v=\sum_{i=1}^d a_i$,
$S=\prod_{i=1}^{d-1}s_i$, and $\sigma=\sum_{i=1}^{d-1}s_i$.
Suppose
$A<\frac{S\sigma}{6}-(7S\sigma^4)^{1/3}
-\min\{S-1,\sigma-d+1\}-2.$
If $s_d>3^{(d+1)S+o_d(S)}$, then
\[
\ex(n,\mathcal K^{(d)}_{a_1,\dots,a_d},
\mathcal K^{(d)}_{s_1,\dots,s_d})
=\Theta_{a_1,\dots,a_d,s_1,\dots,s_d}(n^{v-A/S}).
\]
\end{corollary}

The symmetric construction also applies when the counted graph is a
clique.  Combining Theorem~\ref{thm:generalized Turan hypergraph explicit}
with the upper bound of Alon and Shikhelman~\cite{alon2016many}
gives the following.

\begin{corollary}\label{cor:generalized Turan}
Suppose that $m, s, t$ are positive integers such that $s+ 1 \ge m \ge 2$ and $t > 27^{ (1 + o(1))s}$. If $\binom{m}{2} <  s^2/6 - 7^{1/3}s^{5/3}-s-1$, then
\[
\ex(n, \mathcal{K}_m, \mathcal{K}_{s,t}) = \Theta_{m,s,t} (n^{m- m(m-1)/(2s)}).
\]
\end{corollary}

\subsection{Diagonal Multicolor Ramsey Numbers}
A dense forbidden-subgraph-free hypergraph need not come with a
cover of all edges of the complete hypergraph by comparably few
isomorphic copies.  The symmetric framework supplies something
stronger: a two-dimensional space of forms in which every nonzero
form defines a forbidden-subgraph-free hypergraph.  Its projective
line then gives a coloring with only $q+1$ colors.

For a $d$--uniform hypergraph $\mathcal H$, its \emph{diagonal
multicolor Ramsey number} is defined by
\[
\rrr_K^{(d)}(\mathcal H)
\coloneqq
\min\bigl\{N:\text{every $K$--coloring of $\mathcal{K}_N^{(d)}$ contains a
monochromatic copy of $\mathcal H$}\bigr\}.
\]
We omit the superscript when $d=2$. For graphs, Chung and Graham \cite{CG75} proved that $\rrr_K(\mathcal{K}_{s,t}) = O_{s,t}(K^s)$ for any $K \ge 2$ and $t \ge s \ge 2$. They conjectured that this is optimal. 

\begin{conjecture}\label{conj:diagonal ramsey}
For any $t \ge s \ge 2$ and $K \ge 2$, it holds that $\rrr_K(\mathcal{K}_{s,t}) = \Theta_{s,t}(K^s)$.
\end{conjecture}
Alon, R\'{o}nyai and Szab\'{o}~\cite[Theorem~8]{Alon1999}
proved that Conjecture~\ref{conj:diagonal ramsey} is true whenever $t\ge (s-1)!+1$. Lazebnik and Mubayi~\cite{LazebnikMubayi2002} obtained the corresponding
$d$--uniform result for $\mathcal{K}^{(d)}_{1,\dots,1,s,t}$ under the
same factorial-type condition. Our first Ramsey application concerns the lower bound of $\rrr_K^{(d)} (\mathcal{F}(t))$. 

\begin{theorem}[Diagonal Multicolor Ramsey Number]\label{thm-ramsey_1}
Let $\mathcal F$ be a $(d-1)$--partite $(d-1)$--uniform hypergraph with
no isolated vertices. Put $S\coloneqq |E(\mathcal F)|$ and $\sigma\coloneqq |V(\mathcal F)|$. Suppose that $S$ is sufficiently large in terms of $d$. For any $t > 3^{2S+\rho+\sigma+d+3}T^{1+\log\rho}\rho!$ where $\rho\coloneqq \lfloor(7S\sigma)^{1/3} \rfloor$ and $T\coloneqq(S+\rho)\sigma+3$, we have
\[
\rrr_K^{(d)} (\mathcal{F}(t))
=\Omega_{\mathcal{F},t}(K^S).
\]
\end{theorem}

If the base is Sidorenko, Lee's extremal upper bound and a count of
the edges in the color classes give the reverse Ramsey inequality.
Thus the same Sidorenko condition yields sharp exponents in both
extremal and diagonal Ramsey problems.

\begin{corollary}\label{cor:diagonal Ramsey specializations}
If $\mathcal{F}$ is Sidorenko and the hypotheses of Theorem~\ref{thm-ramsey_1} hold, then
\[
\rrr_K^{(d)}\!\left(\mathcal F(t)\right)
=\Theta_{\mathcal F,t}(K^S).
\]
In particular, we have $\rrr_K^{(d)} (
\mathcal K^{(d)}_{s_1,\dots,s_{d}} ) = \Theta_{d,s_1,\dots,s_{d}}(K^S)$ if $S \coloneqq \prod_{i=1}^{d-1} s_i$ is sufficiently large in terms of $d$ and $s_d > 3^{(d+1+ o_d(1)) S}$.
\end{corollary}
For graphs, this replaces the factorial threshold of
Alon, R\'onyai and Szab\'o~\cite{Alon1999} by an exponential one
for all sufficiently large $s$, proving
Conjecture~\ref{conj:diagonal ramsey} in that range.
In higher uniformity, it both replaces the factorial threshold and
removes the restriction $s_1=\cdots=s_{d-2}=1$ in the result of
Lazebnik and Mubayi~\cite{LazebnikMubayi2002}.

\subsection{Off-Diagonal Multicolor Ramsey Numbers}\label{subsec:off-diagonal Ramsey}
For graphs $\mathcal{H}_1,\dots, \mathcal{H}_{K+1}$, we denote by $\rrr(\mathcal{H}_1,\dots, \mathcal{H}_{K+1})$ the usual multicolor Ramsey number, and write
\[
\rrr_K(\mathcal{H}; \mathcal{K}_n)
\coloneqq
\rrr(\underbrace{\mathcal{H},\dots,\mathcal{H}}_{K\text{ copies}}, \mathcal{K}_n).
\]
For $K=1$, this is the usual off-diagonal Ramsey number.
When several colors must avoid $\mathcal H$ and the last color must
avoid a clique, edge density is not enough: the union of the first
color classes must also meet every large vertex set.
For fixed $s\ge2$, $K\ge3$ and $t\ge(s-1)!+1$,
Alon and R\"odl~\cite[Theorem~3.7]{alon2005sharp} proved
$\rrr_K(\mathcal K_{s,t};\mathcal K_n)
=\Theta_{s,t,K}((n/\log n)^s)$ by using pseudorandom norm graphs.
Our construction retains the local density needed for this argument
at an exponential, rather than factorial, threshold for $t$.

\begin{theorem}[Off-Diagonal Multicolor Ramsey Number]\label{thm:Ramsey}
There is an absolute constant $C$ such that, for all integers $s\ge2$, $K\ge3$ and $t>C9^s s^{4s^{2/3}}$, we have
\[
\rrr_K(\mathcal K_{s,t};\mathcal K_n)
=\Theta_{s,t,K}\!\Big(\frac{n^s}{\log^s n}\Big).
\]
\end{theorem}

The upper bound follows from Alon and R\"odl's extremal-to-Ramsey
estimate~\cite[Lemma~3.3]{alon2005sharp}.  For the lower bound, we
obtain $\mathcal K_{s,t}$-free graphs whose adjacency matrices have
controlled centered spectral norm.  A container argument bounds
the number of large independent sets, and random relabelings of
these graphs supply the required color classes.

Table~1 compares our bounds with selected previous results; each entry is subject to the hypotheses stated in Sections~1.1--1.5.

\begin{table}[htbp]
    \centering
    \small
    \setlength{\tabcolsep}{4pt}
    \renewcommand{\arraystretch}{1.2}
    \begin{tabularx}{\textwidth}{
        |>{\raggedright\arraybackslash}X
        |>{\raggedright\arraybackslash}X
        |>{\raggedright\arraybackslash}X
        |>{\raggedright\arraybackslash}X|
    }
        \hline
        Expected Result & Existing Threshold & Our Threshold & Improvement \\ \hline
        $\ex(n,\mathcal{K}^{(d)}_{\mathbf{s}}) = \Theta(n^{d - \frac{1}{S}})$ & \cite{pohoata2021norm}: $s_d > ((d-1) (S - 1))!$ &  $s_d > 9^{(1 + o_d(1))S}$  & resolve Conj.~\ref{conj:Turan} in exponential range \\ \hline
         $\ex(n,\mathcal{H}(t)) = \Theta(n^{d - \frac{1}{\sss(\mathcal{H})}})$ & none & $\mathcal{H}$ Sidorenko and $t > 9^{(1 + o_d(1))|E(\mathcal{H})|}$  & resolve Conj.~\ref{conj:Lee} in exponential range \\ \hline        
         $\zz(m,n,s,t) = \Omega(mn^{1- \frac{1}{s}})$ & \cite{Bukh_2024}: $\log_n m = O_s(t^{\frac{1}{1 + 2 \log s}})$ & Cor.~\ref{cor:sided-Zarankiewicz-complete}: $\log_n m \le C t^{\frac{1}{2 + \log s - \frac{1}{s+1}}}$ & answer Question~\ref{question:Zarankiewicz} in broader range \\ \hline
         $\ex(n, \mathcal{K}^{(d)}_{\mathbf{a}}, \mathcal{K}^{(d)}_{\mathbf{s}}) = \Theta \big( n^{\sum_{i=1}^d a_i - \frac{\prod_{i=1}^d a_i}{S}} \big)$ & \cite{MYZ18}: $s_d > C$ for some unknown $C$ & Cor.~\ref{cor:generalized Turan-1}: $s_d > 3^{(d + 1 +o_d(1))S }$ & quantitative version with exponential range \\ \hline
         $\rrr_K^{(d)} (\mathcal K^{(d)}_{\mathbf{s}} ) = \Theta(K^S)$ & \cite{LazebnikMubayi2002}: $s_1 = \cdots = s_{d-2} = 1$ and $s_d > (s_{d-1}-1)!$ & Cor.~\ref{cor:diagonal Ramsey specializations}: $S \gg d$ and $s_d > 3^{(d+1+ o_d(1)) S}$ & resolve Conj.~\ref{conj:diagonal ramsey} in exponential range \\ \hline
         $\rrr_K(\mathcal{K}_{s,t}; \mathcal{K}_n)
=\Theta\big(\frac{n^s}{\log^s n} \big)$ & \cite{alon2005sharp}: $K \ge 3$ and $t > (s-1)!$ & Thm.~\ref{thm:Ramsey}: $K\ge3$ and $t> 9^{(1 + o(1))s}$ & factorial range to exponential range \\ \hline
    \end{tabularx}
    \caption{Comparison of our thresholds with existing ones. Here $\mathbf{s} = (s_1,\dots, s_d)$, $\mathbf{a} = (a_1,\dots, a_d)$ and $S = s_1 \cdots s_{d-1}$.}
    \label{old:25:tab:comparison}
\end{table}

\subsection{Our approach}\label{subsec:ideas}
The common geometric object is a partial evaluation.  Start with a
variety over $\mathbb F_q$ and a form $F$ that is homogeneous of degree
$m$ in each of $d$ blocks of variables.  Declare a $d$-tuple to be an
edge when $F$ vanishes.  For an embedding of a base $\mathcal H$,
the common apex neighborhood is cut out by the $S=|E(\mathcal H)|$
forms $F(v_1,\dots,v_{d-1},\mathord\cdot)$ associated with its edges.
A vertex set of order $q^S$ and edge density $q^{-1}$ gives precisely
the exponent $d-1/S$.  The difficulty is to make all these common
neighborhoods finite simultaneously while retaining the expected
number of rational points.

\paragraph{Interpolation and common slicing.}
We use Bukh's $m$-independent varieties~\cite{Bukh_2024}: any sufficiently
small set of their rational points has linearly independent images
under the degree-$m$ Veronese map.  This is exactly the interpolation
property needed to turn distinct partite partial evaluations into
independent random homogeneous polynomials.  One first works on a
variety of dimension $S+\rho$, where the relevant $S$ equations cut
the dimension down to $\rho$ with sufficiently high probability.
A further sequence of hypersurfaces cuts all these common-link
varieties to dimension zero, while leaving order $q^S$ rational
points in the vertex variety.  B\'ezout's theorem then bounds every
apex neighborhood by the same degree product.  Keeping the initial
degree equal to three makes this product exponential in $S$ for
bases without isolated vertices and fixed uniformity, with only a
subexponential cost from the final slices.

\paragraph{The obstruction created by symmetry.}
For constructions on one vertex set, $F$ must be symmetric.  Its
partial evaluations are then not independent: two base edges that
differ in one vertex can use the same coefficient.  The essential
point is that these dependencies are sparse.  After earlier partial
evaluations have been exposed, the next one remains uniform on an
affine space whose codimension is at most
$\min\{S-1,\sigma-d+1\}$, where $\sigma$ is the number of base
vertices.  A Hilbert-function bound controls the probability that
this affine space produces a form vanishing on an entire component.
This yields uniform dimension drop despite the symmetry constraint.

\paragraph{From a good form to a good linear space.}
For counting and coloring, one good form is insufficient.  We show
that the exceptional forms are sparse enough to avoid an entire
prescribed-dimensional linear subspace, and then choose one slice
that works for every nonzero form in that subspace and every base
embedding.  A space of dimension $e(\mathcal F)+1$ leaves a nonzero
solution after imposing the equations of any prescribed copy of
$\mathcal F$, giving the expected generalized Tur\'an exponent.
A space of dimension two serves a different purpose: every edge
evaluation has a nonzero kernel, and a projective point in that
kernel supplies a color.
This is the link between the counting and diagonal Ramsey applications.

\paragraph{Limited independence and local density.}
The off-diagonal application requires a different use of the
symmetric construction.  We choose the form and the slices randomly,
rather than restrict to a fixed good subspace.  Evaluations on an
edge family spanning sufficiently few vertices are jointly
independent.  This finite independence suffices for the trace
moments of the centered adjacency matrix; independence of the
slicing equations supplies the corresponding vertex-survival
probabilities.  The resulting spectral event controls all large
vertex subsets simultaneously.  Consequently, the algebraic
construction provides not only forbidden-subgraph avoidance but
also the independent-set estimates needed by the
Alon--R\"odl method.

\subsection{Organization}

Section~\ref{sec:prelim} records the geometric preliminaries, and
Section~\ref{sec:multilinear-framework} develops the interpolation,
symmetric dimension-drop, and simultaneous slicing results.
Sections~\ref{sec:Turan}--\ref{sec:algebraic-coloring} give the
Tur\'an, sided Zarankiewicz, generalized Tur\'an, and diagonal Ramsey
applications.  Section~\ref{sec:pseudorandom-Ramsey} proves the
off-diagonal Ramsey theorem, using the spectral estimates in
Appendix~\ref{appendix:spectral}.  Section~\ref{sec:concluding}
discusses the principal remaining questions.

\section{Preliminaries}\label{sec:prelim}
This section records definitions and lemmas which are necessary for our algebraic constructions. In this paper, we use $\mathbb{F}$ for an arbitrary field. For a prime power $q$, $\mathbb{F}_q$ denotes the finite field with $q$ elements. We write $\mathbb{P}^N$ for the $N$-dimensional projective space over $\overline{\mathbb{F}}$, and $\mathbb{P}^N(\mathbb{F})$ for the $N$-dimensional projective space over $\mathbb{F}$. 

We fix a nonzero vector representative whenever a projective point
is evaluated by a homogeneous form.  Vanishing is independent of
this choice.  Dimensions and irreducible components are taken over
the algebraic closure; independence of rational points is over the
base field.  A maximum over an empty indexing set is $-\infty$.

\begin{lemma}\cite[Lemma~3]{Bukh_2024}\label{generalized schwarz-zippel}
If $\mathsf{W} \subseteq \mathbb{P}^{N}$ is a subvariety of dimension $k$ and degree $D$, then $\vert \mathsf{W}(\mathbb{F}_{q})\vert\le D(q^{k+1} - 1)/(q-1)$.
\end{lemma}

We consider the Veronese map $\nu_m: \mathbb{F}^{N+1} \to  \mathbb{F}^{\binom{N + m}{m}}$ defined by 
\begin{equation}\label{eq:veronese}
\nu_m (u_{0},\dots,u_{N}) \coloneqq (u_{i_1}\cdots u_{i_m})_{0 \le i_1 \le \cdots \le i_m \le N}. 
\end{equation}
Given a finite subset $X = \{ [v_1],\dots,  [v_s]\} \subseteq \mathbb{P}^N(\mathbb{F})$, we define 
\[
h_X(m) = \dim \spa_{\mathbb{F}} \left\lbrace
\nu_m(v_1),\dots,  \nu_m(v_{s})
\right\rbrace.
\]
The following notion is introduced in \cite{Bukh_2024}.
\begin{definition}[$s$-wise $m$-independence]\label{def-m-independent}
Let $s,m \ge 0$ be fixed integers. A subset $X \subseteq \mathbb{P}^N(\mathbb{F})$ is \emph{$s$-wise $m$-independent} if $h_S(m) = s$ for any $S \subseteq X$ with $|S| = s$. If $X \subseteq \mathbb{P}^N(\mathbb{F})$ is finite, then $X$ is said to be \emph{minimally $m$-dependent} if $X$ is not $|X|$-wise $m$-independent, but any proper subset $Y \subsetneq X$ is $|Y|$-wise $m$-independent.
\end{definition}

Given integers $N,t, m \ge 0$, we define 
\begin{align*}
    \mathcal{X}(N,t,m)&\coloneqq \{ ([v_1],\dots, [v_t]) \in (\mathbb{P}^{N})^t: \{[v_1],\dots, [v_t]\} \text{~is a minimally $m$-dependent $t$-set} \},\\
    \psi(N,t,m)&\coloneqq \dim \mathcal{X}(N,t,m).
\end{align*}

For ease of reference,  we record the lemma that estimates the value of $\psi(N,t,m)$. 
\begin{lemma}[Upper Bound of $\psi(N,t,m)$]\cite[Lemma~22]{Bukh_2024}\label{m-independent quant}
Given integers $N,t,m\ge 3$,  we have 
\begin{enumerate}[\normalfont(i)]
\item If $t \le m +1$,  then $\mathcal{X}(N,t,m)$ is an empty set.     
\item If $m+2\le t \le N$,  then $\psi(N,t,m) \le \lfloor 3t/(m+4) \rfloor \big( N+1+ (m-2)t/(m+4) \big)$.
\end{enumerate}
\end{lemma}
For completeness, when $m\ge 3$, we also have
\[
    \mathcal{X}(N,2,m)=\varnothing
    \qquad\text{and hence}\qquad
    \psi(N,2,m)=-\infty.
\]
Indeed, the Veronese images of any two distinct projective points are linearly independent. Therefore, in what follows, expressions of the form
\[
    \max_{2\le t\le s}\frac{\psi(N,t,m)}{t-1}
\]
are well defined under our convention that the dimension of the empty set is $-\infty$.

Let $R=\mathbb{F}_q[x_0,...,x_N]$ be the polynomial ring in variables $x_0,\dots, x_N$ over $\mathbb{F}_q$. We denote by $R_m$ the space of degree $m$ homogeneous polynomials in $R$. For $f_{1},\dots,f_{r}\in R$, we denote by $\mathsf{V}(f_{1},\dots,f_{r})$ the variety defined by $f_1,\dots, f_r$. We write $\mathsf{V}(f_{1},\dots,f_{r})(\mathbb{F}_q)$ for the set of common zeros of $f_1$, $\dots$, $f_r$ in $\mathbb{P}^{N}(\mathbb{F}_q)$.

\begin{lemma}[$m$-independent Sets over Large Fields]\cite[Lemma~15]{Bukh_2024}\label{m-independent existence}
There is a function $q_0: \mathbb{N}^4 \to \mathbb{N}$ with the following property. If $N,s,m,r$ are positive integers and $q \ge q_0(N,s,m,r)$ is a prime power such that $\max_{2 \le t \le s} \left\lbrace \psi(N,t,m)/(t-1) \right\rbrace < r < N$,  then there exist $f_{1},\dots,f_{r}\in R_m$ such that $\mathsf{V}(f_{1},\dots,f_{r})$ is of pure dimension $N-r$. Moreover, $\mathsf{V}(f_{1},\dots,f_{r})(\mathbb{F}_q)$ is $s$-wise $m$-independent, and $\vert \mathsf{V}(f_{1},\dots,f_{r})(\mathbb{F}_q)\vert\ge \frac{1}{2}q^{N-r}$.
\end{lemma}

\begin{lemma}\cite[Lemma~4]{Bukh_2024}\label{rational point positive probability}
Let $m_1,\cdots,m_r$ be positive integers and $S \subseteq \mathbb{P}^N(\mathbb{F}_q)$ be a nonempty subset. Then for a uniformly random $(g_{1},\dots,g_r)\in R_{m_{1}}\times\cdots\times R_{m_r}$,
\[
\Pr \!\Big[ \vert S\cap \mathsf{V}(g_1,\dots,g_r)\vert\le \frac{\vert S \vert}{2q^{r}} \Big] \le \frac{4q^r}{\vert S \vert}.
\] 
\end{lemma}

We also consider the function
\[
M: \mathbb{N}^2 \to \mathbb{N},\quad M_{r}(t) \coloneqq \min
\Big\lbrace m\in \mathbb{N}: \binom{m+r}{r}> t 
\Big\rbrace.
\]

\begin{lemma}\cite[Lemma~24]{Bukh_2024}\label{inequality} For any positive integers $t$ and $r$, we have $\prod_{k=1}^{r}M_{k}(t)\le t^{1+log\,r}r!$
\end{lemma}

\begin{lemma}\cite[Proposition~12(b)]{Bukh_2024}\label{lem:regular sequence-probability-1}
Let $0\le r \le k\le N$ be integers and let $\mathsf{W} \subseteq \mathbb{P}^{N}$ be a subvariety of pure dimension $k$ and degree $D$. Assume $t, m_1,\dots, m_r$ are positive integers such that $m_{i}\ge M_{k-i+1}(t)$, $1 \le i \le r$. Then there exists a constant $C \coloneqq C(t,r,D,m_1,\dots, m_r)> 0 $ such that for a uniformly random $(g_{1},\dots,g_r)\in R_{m_{1}}\times\cdots\times R_{m_r}$, it holds that
\[
\Pr\! \big[\dim( \mathsf{W} \cap \mathsf{V}(g_{1},\dots,g_r))>k-r \big] \le Cq^{-t},
\]
\end{lemma}

\section{The Multilinear Algebraic Framework}
\label{sec:multilinear-framework}
To begin with, we establish the interpolation property of $m$-independent sets, which underlies our multilinear algebraic framework. We recall that $R_m \coloneqq \mathbb{F}_q[x_0,...,x_N]_m$ consists of all degree $m$ homogeneous polynomials in variables $x_0,\dots, x_N$ with coefficients in $\mathbb{F}_q$.

\begin{lemma}[Interpolation of $m$-independent Sets]
If $\{ [v_{1}],\dots, [v_{s}] \} \subseteq \mathbb{P}^{N}(\mathbb{F}_q)$ is $m$-independent, then there exists an $\mathbb{F}_q$-basis $f_1,\dots, f_{\binom{N+m}{m}}$ of $R_m$ such that $f_i(v_j) = \delta_{ij}$ for each $1 \le i \le \binom{N+m}{m}$ and $1 \le j \le s$.  Here $\delta$ is the Kronecker delta function.
\label{characterization of m-independence}
\end{lemma}
\begin{proof}
 For simplicity, we denote $M \coloneqq \binom{N+m}{m}$. Let $\nu_m: \mathbb{F}_q^{N+1} \to \mathbb{F}_q^M$ be the Veronese map defined in \eqref{eq:veronese}. By the definition of $m$-independent sets, $\nu_m(v_1),\dots, \nu_m(v_s)$ are linearly independent. We extend $\nu_m(v_1),\dots, \nu_m(v_s)$ to an $\mathbb{F}_q$-basis of $\mathbb{F}_q^M$ and let $\ell_1,\dots, \ell_M$ be its dual basis. Taking $f_j \coloneqq \ell_j \circ \nu_m$, we obtain a desired basis of $R_m$.
\end{proof}

\subsection{Basic Properties of Partial Evaluations on Varieties}
Given a positive integer $d$, we denote 
\[
R_m^{\otimes d} \coloneqq \underbrace{R_m \otimes \cdots \otimes R_m}_{d \text{~times}}.
\]
By definition, $R_m^{\otimes d}$ consists of polynomials in $\mathbb{F}_q[x_{1,0},\dots, x_{1,N},\dots, x_{d,0},\dots, x_{d,N}]$ which are homogeneous of degree $m$ in variables $x_{i,0},\dots, x_{i,N}$ for each $1 \le i \le d$. For $v_1,\dots, v_{d-1} \in \mathbb{F}_q^{N+1}$ and $F \in R_m^{\otimes d}$, we write $F(v_1,\dots, v_{d-1}, \cdot)$ for the element in $R_m$ obtained by evaluating $F$ at $(v_1,\dots, v_{d-1}) \in (\mathbb{F}_q^{N+1})^{d-1}$.

\begin{proposition}[Independence]\label{prop:lem-local independent}
Let $d, m, N$ be positive integers, and let $\{[v_{i,1}],\cdots,[v_{i,s_i}]\} \subseteq \mathbb{P}^N(\mathbb{F}_q)$ be $m$-independent for each $1 \le i \le d - 1$. If $F$ is chosen uniformly at random from $R_m^{\otimes d}$, then the set $\{F(v_{1,j_1},\cdots,v_{d-1,j_{d-1}},\cdot) \in R_m: 1\le j_{i}\le s_{i},\; 1 \le i \le d-1\}$ is independently and uniformly distributed over $R_m$.
\end{proposition}
\begin{proof}
Denote $M \coloneqq \binom{N+m}{m}$. For each $1 \le i \le d - 1$, since $\{[v_{i,1}],\dots, [v_{i,s_i}]\}$ is $m$-independent, Lemma~\ref{characterization of m-independence} implies that there is a basis $\big\{f_{i,j}:1 \le j \le M \big\}$ of $R_m$ such that $f_{i,j}(v_{i,t}) = \delta_{jt}$, where $1 \le j \le M$ and $1 \le t \le s_i$. Let $\{f_{d,1},\dots, f_{d,M}\}$ be an arbitrary basis of $R_m$. Then we may expand $F \in (R_m)^{\otimes d}$ as $F = \sum_{j_1,\dots, j_d =1}^{M} \lambda_{j_1, \cdots,j_d} f_{1,j_1}\otimes \cdots \otimes f_{d, j_d}$ for some independently uniform variables $\lambda_{j_1, \cdots,j_d}$. Consequently, we obtain 
\[
F(v_{1,j_1},\dots, v_{d-1,j_{d-1}},\cdot) = \sum_{j_d = 1}^M \lambda_{j_1,\dots, j_d} f_{d, j_d}
\]
for each $1 \le j_i \le s_i$, $1 \le i \le d-1$. Since $\lambda_{j_1, \cdots,j_d}$'s are independently uniform and $\{f_{d,1},\dots, f_{d,M}\}$ is a basis of $R_m$, the set $\{F(v_{1,j_1},\dots, v_{d-1,j_{d-1}},\cdot)\}$ is independently and uniformly distributed.
\end{proof}

Next, we assume that $d < \ch \mathbb{F}_q$. We write $S^d R_m \subseteq R_m^{\otimes d}$ for the subspace consisting of symmetric tensors:
\[
f_1 \circ \cdots \circ f_d \coloneqq \frac{1}{d!} \sum_{\pi \in \mathfrak{S}_d} f_{\pi(1)} \otimes \cdots \otimes f_{\pi(d)},\qquad f_1,\dots, f_d \in R_m.
\]

Each $F \in S^d R_m$ is identified with a symmetric multihomogeneous form of degree $m$ in each of the $d$ blocks of variables. Given $[v_1],\dots, [v_{d-1}] \in \mathbb{P}^N(\mathbb{F}_q)$, the symmetry of $F$ implies that $F(v_1,\dots, v_{d-1},\cdot) = F(v_{\tau(1)}, \cdots, v_{\tau(d-1)},\cdot)$ for any $\tau \in \mathfrak{S}_{d-1}$. For a positive integer $\sigma$, we denote by $\binom{[\sigma]}{d-1}$ the collection of $(d-1)$-subsets of $[\sigma]$. For fixed $[v_1],\dots, [v_\sigma] \in \mathbb{P}^N(\mathbb{F}_q)$, we denote 
\begin{equation}\label{eq:FI}
F_I \coloneqq F(v_{i_1},\dots, v_{i_{d-1}}, \cdot),
\end{equation}
where $I = \{i_1,\dots, i_{d-1}\} \in \binom{[\sigma]}{d-1}$. 

\begin{proposition}[Regularity with High Probability]\label{prop:symmetric partial evaluations}
Suppose $d < \ch \mathbb{F}_q$. Let $\{[v_1],\dots,[v_\sigma]\} \subseteq \mathbb{P}^N(\mathbb{F}_q)$ be an $m$--independent subset, and let $\mathsf{W} \subseteq \mathbb{P}^N$ be a subvariety of pure dimension $k$ and degree $D$. If $F$ is chosen uniformly from $S^d R_m$, then for each $\mathcal{I} \subseteq \binom{[\sigma]}{d-1}$ with $|\mathcal{I}| = S \le k$, we have
\[
\Pr\!\big[
\dim (\mathsf{W} \cap \mathsf{V} ( F_I: I \in \mathcal I) ) > k-S \big] \le SDm^{S-1} q^{\min\{ S-1, \sigma - d + 1 \} - \binom{m+k-S + 1}{m}}.
\]
\end{proposition}
\begin{proof}
Put $M=\dim R_m$.  By Lemma~\ref{characterization of m-independence},
choose a basis $f_1,\dots,f_M$ with $f_a(v_i)=\delta_{ai}$ for
$i\in[\sigma]$.  In the symmetric tensor basis, write
\begin{equation}\label{prop:symmetric partial evaluations:eq1}
F=\sum_{1\le a_1\le\cdots\le a_d\le M}
\lambda_{a_1,\dots,a_d}f_{a_1}\circ\cdots\circ f_{a_d},
\end{equation}
where all coefficients are independent and uniform.
For $I\in\mathcal I$, the coefficients of $F_I$ are nonzero scalar
multiples of the $\lambda$'s indexed by the multisets
$\Lambda(I)=\{I\sqcup\{a\}:a\in[M]\}$.
The scalar multiples are nonzero because
$\operatorname{char}\mathbb F_q>d$.
For distinct $I,I'$, the sets $\Lambda(I)$ and $\Lambda(I')$ have
at most one common element.  Such an element is necessarily
$I\cup\{a\}$ with $a\in[\sigma]\setminus I$.

Order $\mathcal I$ as $I_1,\dots,I_S$ and set
$\mathsf X_j=\mathsf W\cap\mathsf V(F_{I_1},\dots,F_{I_j})$.
Expose all coefficients occurring in the first $j$ evaluations.
Conditioned on them, $F_{I_{j+1}}$ is uniform on an affine subspace
of $R_m$ of codimension
$t_j\le\min\{j,\sigma-d+1\}$.
Before the first failure of dimension drop, $\mathsf X_j$ is pure
of dimension $k-j$ and has degree at most $Dm^j$.

Let $Y$ be a geometric irreducible component of $\mathsf X_j$.
The Hilbert-function bound~\cite[Lemma~6]{Bukh_2024} shows that the
space of degree-$m$ forms vanishing on $Y$ has codimension at least
$\binom{m+k-j}{m}$ over the algebraic closure.  Its intersection
with $R_m$ has at least the same codimension over $\mathbb F_q$:
linearly independent vectors over $\mathbb F_q$ remain independent
after extending the field.  This applies even when $Y$ is not
defined over $\mathbb F_q$.  Hence, conditioned on the exposed
coefficients,
\[
\Pr[F_{I_{j+1}}|_Y=0]
\le q^{t_j-\binom{m+k-j}{m}}.
\]
There are at most $Dm^j$ components to consider.  Summing over the
possible first failures gives
\[
\Pr[\dim\mathsf X_S>k-S]
\le\sum_{j=0}^{S-1}Dm^j
q^{\min\{j,\sigma-d+1\}-\binom{m+k-j}{m}}
\le SDm^{S-1}
q^{\min\{S-1,\sigma-d+1\}-\binom{m+k-S+1}{m}}.
\qedhere
\]
\end{proof}

\subsection{Simultaneous Regularity of Partial Evaluations on Varieties}
The primary purpose of this subsection is to establish a simultaneous regularity result for partial evaluations.
\begin{lemma}
\label{lem:general-simultaneous-slicing}
Let $N,k,D,E$ be positive integers and let $b, \rho$ be nonnegative integers with $\rho < k$. Suppose that $q$ is sufficiently large and $\mathsf{W} \subseteq \mathbb{P}^N$ is a pure $k$--dimensional variety with $\deg \mathsf{W} \le D$ and $|\mathsf{W}(\mathbb{F}_q)| \ge q^k/2$. Let $\mathfrak{Y}$ be a nonempty finite collection of pure $\rho$--dimensional subvarieties of $\mathsf{W}$ whose degree is at most $E$. If $|\mathfrak{Y}|\le C q^b$ for some constant $C$ independent of $q$, then for any $\delta_i \ge M_{\rho-i+1}(b+1)$, $i \in [\rho]$, a randomly and independently chosen sequence $(g_1,\dots, g_{\rho}) \in R_{\delta_1} \times \cdots \times R_{\delta_\rho}$ satisfies the following two conditions with probability $1-O_{C,b,\rho, D, E, \delta_1,\dots, \delta_\rho}(q^{-1})$:
\begin{enumerate}[\normalfont(i)]
\item $\mathsf{X} \coloneqq
\mathsf{W} \cap \mathsf{V}(g_1,\dots,g_\rho)$ is of pure dimension $k-\rho$ and $|\mathsf{X}(\mathbb{F}_q)| \ge\frac{1}{4}q^{k-\rho}$. \label{lem:general-simultaneous-slicing-item1}
\item $\dim ( \mathsf Y\cap\mathsf V(g_1,\dots,g_\rho) ) = 0$ for each $\mathsf Y\in\mathfrak Y$. \label{lem:general-simultaneous-slicing-item2}
\end{enumerate}
In particular, a sequence $(g_1,\dots, g_{\rho}) \in R_{\delta_1} \times \cdots \times R_{\delta_\rho}$ satisfying \ref{lem:general-simultaneous-slicing-item1} and \ref{lem:general-simultaneous-slicing-item2} exists.
\end{lemma}
\begin{proof}
There is nothing to prove for $\rho = 0$. We may therefore assume that $\rho\ge1$. We set $\mathsf X \coloneqq \mathsf W\cap\mathsf V(g_1,\dots,g_\rho)$. Since $k \ge \rho$ and $M_r(t)$ is a nonincreasing function of $r$, we have
\[
\delta_i
\ge M_{\rho-i+1}(b+1)
\ge M_{k-i+1}(b+1)
\]
for every $i \in [\rho]$. 

By Lemma~\ref{lem:regular sequence-probability-1}, we obtain
\[
\Pr\left[ \dim\mathsf X > k-\rho \right] = O_{b,\rho,D,\delta_1,\dots, \delta_\rho}(q^{-b-1}).
\]
By construction, each irreducible component of $\mathsf X$ has dimension at least
$k-\rho$. Therefore,
\begin{equation}\label{lem:general-simultaneous-slicing-eq1}
\Pr\left[ \mathsf{X} \text{~is not of pure dimension $k - \rho$~} \right] = O_{b,\rho,D,\delta_1,\dots, \delta_\rho}(q^{-b-1}). 
\end{equation}
On the other hand, Lemma~\ref{rational point positive probability} also shows that 
\begin{equation}\label{lem:general-simultaneous-slicing-eq2}
\Pr\left[ |\mathsf X(\mathbb F_q)|<\frac14q^{k-\rho}
\right] \le
\Pr\left[ |\mathsf X(\mathbb F_q)|
\le \frac{|\mathsf W(\mathbb F_q)|}{2q^\rho}
\right] \le \frac{4q^\rho}{|\mathsf W(\mathbb F_q)|} \le
8q^{-(k-\rho)}.
\end{equation}

Next, for each $\mathsf Y\in\mathfrak Y$, it is of pure dimension $\rho$ and degree at most $E$. Since $\delta_i\ge M_{\rho-i+1}(b+1)$, Lemma~\ref{lem:regular sequence-probability-1} yields
\[
\Pr\left[
\dim\bigl(
\mathsf Y\cap\mathsf V(g_1,\dots,g_\rho)
\bigr)>0
\right]
=
O_{b,\rho, E, \delta_1,\dots, \delta_\rho}(q^{-b-1}).
\]
Thus, $|\mathfrak Y|\le Cq^b$ implies
\begin{equation}\label{lem:general-simultaneous-slicing-eq3}
\Pr\left[ \dim ( \mathsf Y\cap\mathsf V(g_1,\dots,g_\rho) )>0
\text{~for some $\mathsf Y\in\mathfrak Y$~}
\right] = O_{C,b,\rho, E, \delta_1,\dots, \delta_\rho}(q^{-1}).   
\end{equation}

Combining \eqref{lem:general-simultaneous-slicing-eq1}--\eqref{lem:general-simultaneous-slicing-eq3}, we conclude that \ref{lem:general-simultaneous-slicing-item1} and \ref{lem:general-simultaneous-slicing-item2} simultaneously hold with probability $1-O_{C,b,\rho, D, E, \delta_1,\dots, \delta_\rho}(q^{-1})$.
\end{proof}

We denote by $\Gr(\ell, \mathbb{F}_{q}^{n})$ the Grassmannian of $\ell$-dimensional linear subspaces of $\mathbb{F}_{q}^{n}$, and recall the following explicit formula for its cardinality.

\begin{lemma}\cite[Proposition~1.7.2]{stanley2011enumerative}\label{counting subspace} 
For any positive integers $\ell \le n$, we have 
\[
\left\lvert \Gr(\ell,  \mathbb{F}_{q}^{n}) \right\rvert  = \frac{(q^{n}-1)\cdots(q^{n}-q^{\ell-1})}{(q^{\ell}-1)\cdots(q^{\ell}-q^{\ell-1})}.
\]
\end{lemma}

For a sequence $[v_1],\dots,[v_\sigma]$ of projective points, each
$I=\{i_1,\dots,i_{d-1}\}\in\binom{[\sigma]}{d-1}$, and each $F\in S^dR_m$, we recall from \eqref{eq:FI} that
$F_I\coloneqq F(v_{i_1},\dots,v_{i_{d-1}},\mathord\cdot)$.

\begin{lemma}\label{lem:symmetric good subspace}
Let $d\ge2$, let $k,\alpha,S,\sigma$ be positive integers satisfying
\[
\alpha\le \binom{m+k- S +1}{m}-k \sigma- \min\{S-1,\sigma-d+1\} -1,\qquad S \le k,
\]
and suppose that $S\le\binom{\sigma}{d-1}$. Suppose that $\mathsf{W} \subseteq \mathbb{P}^N$ is a pure $k$--dimensional,
$\sigma$--wise $m$--independent variety of degree $D$. If $q$ is sufficiently large with $\ch \mathbb{F}_q > d$, then there exists an $\alpha$--dimensional subspace
$\mathcal{U} \subseteq S^d R_m$ satisfying the following property. For each $F \in\mathcal {U} \setminus \{0\}$, every sequence $[v_1],\dots,[v_\sigma]$ of pairwise distinct points in $\mathsf W(\mathbb F_q)$, and every $\mathcal I\subseteq\binom{[\sigma]}{d-1}$ with $|\mathcal I|=S$, we have
\[
\dim\!\big( \mathsf{W} \cap \mathsf{V} (F_I: I \in \mathcal I) \big) = k-S.
\]
\end{lemma}
\begin{proof}
Denote $t \coloneqq \min\{S-1,\sigma-d+1\}$, $M \coloneqq \binom{N+m}{m}$ and $L \coloneqq \binom{M + d - 1}{d}$. Let $\mathcal{B}$ be the set of $F \in S^d R_m \setminus \{0\}$ such that
\[
\dim\!\big( \mathsf{W} \cap \mathsf{V} (F_I: I \in \mathcal I) \big) > k-S.
\]
for some sequence $[v_1],\dots,[v_\sigma]$ of pairwise distinct points in $\mathsf W(\mathbb F_q)$ and some $\mathcal I\subseteq\binom{[\sigma]}{d-1}$ with $|\mathcal I|=S$. By Lemma~\ref{generalized schwarz-zippel}, we have $n \coloneqq |\mathsf{W} (\mathbb{F}_q)|\le D(q^{k+1} - 1)/(q-1)$. Since the number of possible choices for $\mathcal I$ is $C\coloneqq\binom{\binom{\sigma}{d-1}}{S}$, Proposition~\ref{prop:symmetric partial evaluations} implies that
\begin{equation}\label{lem:symmetric good subspace:eq1}
|\mathcal{B}| \le q^L C SDm^{S-1} q^{t - \binom{m+k-S+1}{m}} \Big( \frac{D(q^{k+1} - 1)}{q-1} \Big)^\sigma
\le 2^\sigma C S D^{\sigma + 1} m^{S-1} q^{L + t - \binom{m+k-S+1}{m} + k\sigma} < q^{L-\alpha}.
\end{equation}
Here the last inequality holds for sufficiently large $q$.

Choose $\mathcal U$ uniformly from $\Gr(\alpha,S^dR_m)$.
Each nonzero vector belongs to $\mathcal U$ with probability
$(q^\alpha-1)/(q^L-1)$.  By \eqref{lem:symmetric good subspace:eq1},
\[
\mathbb E|\mathcal U\cap\mathcal B|
=|\mathcal B|\frac{q^\alpha-1}{q^L-1}<1.
\]
Thus some $\alpha$-dimensional subspace avoids $\mathcal B$.
\end{proof}

\begin{proposition}[Simultaneous Regularity]
\label{prop:symmetric choice of parameters}
Let $N,m,Z,\alpha,d,S,\sigma$ be positive integers satisfying
\begin{align*}
d &\ge 2, \qquad S\le\binom{\sigma}{d-1},\qquad N \ge \sigma,\qquad
N>Z>\max_{2\le u\le \sigma} \frac{\psi(N,u,m)}{u-1},\qquad N-Z \ge S, \\
\alpha &\le
\binom{m+N-Z-S+1}{m}-(N-Z)\sigma-\min\{S-1,\sigma-d+1\}-1.
\end{align*}
For any sufficiently large prime power $q$ with
$\operatorname{char}\mathbb F_q>d$, there exist $f_1,\dots,f_Z\in R_m$, an $\alpha$--dimensional subspace $\mathcal U\subseteq S^d(R_m)$, and $g_i \in R_{M_{N-Z-S-i+1}((N-Z)\sigma+\alpha+1)}$, $i \in [N-Z-S]$, such that the following hold.
\begin{enumerate}[\normalfont(i)]
\item \label{prop:symmetric choice of parameters:item1}
The variety $\mathsf W\coloneqq\mathsf V(f_1,\dots,f_Z)$ is of pure dimension $N-Z$ and degree at most $m^Z$. Moreover, it is $\sigma$--wise $m$--independent over
$\mathbb F_q$ and $|\mathsf W(\mathbb F_q)|\ge\frac{1}{2}q^{N- Z}$.
\item \label{prop:symmetric choice of parameters:item2}
The variety $\mathsf X\coloneqq
\mathsf W\cap\mathsf V(g_1,\dots,g_{N-Z-S})$ has dimension $S$ and $|\mathsf X(\mathbb F_q)|\ge\frac14q^S$.
\item \label{prop:symmetric choice of parameters:item3}
For every $F\in\mathcal U\setminus\{0\}$, every sequence
$[v_1],\dots,[v_\sigma]$ of pairwise distinct points in $\mathsf W(\mathbb F_q)$, and every
$\mathcal I\subseteq\binom{[\sigma]}{d-1}$ with $|\mathcal I|=S$, we have
\begin{align}
\dim( \mathsf W\cap \mathsf V(F_I: I \in \mathcal I)) &= N- Z - S, \label{prop:symmetric choice of parameters:item3:eq1} \\
\dim(\mathsf X\cap \mathsf V(F_I: I \in \mathcal I)) &=0. \label{prop:symmetric choice of parameters:item3:eq2}
\end{align}
\end{enumerate}
\end{proposition}
\begin{proof}
For simplicity, we denote 
\begin{align}
\tau&\coloneqq\min\{S-1,\sigma-d+1\}, \\
k&\coloneqq N-Z, \qquad \rho\coloneqq k-S, \qquad T\coloneqq k\sigma+\alpha+1, \qquad \delta_i\coloneqq M_{\rho-i+1}(T),\; i \in [\rho].
\end{align}

By Lemma~\ref{m-independent existence}, there exist $f_1,\dots,f_Z \in R_{m}$ satisfying \ref{prop:symmetric choice of parameters:item1}. Since $\deg\mathsf W\le m^Z$, Lemma~\ref{lem:symmetric good subspace}
supplies an $\alpha$--dimensional subspace
$\mathcal U\subseteq S^d(R_m)$ for which \eqref{prop:symmetric choice of parameters:item3:eq1} holds. For each $F \in \mathcal{U} \setminus \{0\}$, every sequence $[v_1],\dots,[v_\sigma]$ of pairwise distinct points in $\mathsf W(\mathbb F_q)$, and every $\mathcal I\subseteq\binom{[\sigma]}{d-1}$ with $|\mathcal I|=S$, we take $\mathsf Y_{F,v_1,\dots,v_\sigma,\mathcal I} \coloneqq \mathsf W\cap \mathsf V (F_I:I \in \mathcal I)$. Then $\mathsf Y_{F,v_1,\dots,v_\sigma,\mathcal I}$ is of pure dimension $\rho$ and degree at most $m^{Z+S}$.

By Lemma~\ref{generalized schwarz-zippel}, we have $n \coloneqq |\mathsf W(\mathbb F_q)|=O_{m,N,Z}(q^k)$, which implies that the number of choices of $[v_1],\dots,[v_\sigma],\mathcal I$ is at most
\[
\binom{\binom{\sigma}{d-1}}{S}n^\sigma=O_{m,N,Z,d,S,\sigma}(q^{k\sigma}).
\]
We consider
\begin{align*}
\mathfrak{Y} \coloneqq \{\, &
\mathsf Y_{F,v_1,\dots,v_\sigma,\mathcal I}: F \in \mathcal{U} \setminus \{0\},\;
[v_1],\dots,[v_\sigma]\in\mathsf W(\mathbb F_q) \text{~are pairwise distinct},\\
& \mathcal I\subseteq\binom{[\sigma]}{d-1},\; |\mathcal I|=S \,\}.
\end{align*}
Then we have 
\[
|\mathfrak{Y}| \le |\mathcal{U} \setminus \{0\}| O_{m,N,Z,d,S,\sigma}(q^{k\sigma}) = O_{m,N,Z,d,S,\sigma}(q^{k \sigma + \alpha}).
\]
Applying Lemma~\ref{lem:general-simultaneous-slicing}  to $\mathsf{W}$ and $\mathfrak{Y}$ with
\[
b=k\sigma+\alpha, \qquad D=m^Z, \qquad E=m^{Z+S}, \qquad \rho=k-S,
\]
we obtain a sequence $(g_1,\dots, g_{\rho}) \in R_{\delta_1} \times \cdots \times R_{\delta_\rho}$ satisfying \ref{prop:symmetric choice of parameters:item2} and \eqref{prop:symmetric choice of parameters:item3:eq2}.
\end{proof}

\section{Proof of Theorem~\ref{thm-construction of apex partite}}\label{sec:Turan}
Given a $(d-1)$--partite $(d-1)$--uniform hypergraph $\mathcal{H}$ and $t \in \mathbb{N}$, we recall that $\mathcal{H}(t)$ is the $d$--partite $d$--uniform hypergraph whose $d$th part has $t$ vertices, having $\mathcal{H}$ as a common link hypergraph. Let $q_0: \mathbb{N}^4 \to \mathbb{N}$ be the function in Lemma~\ref{m-independent existence}.

\begin{lemma}\label{3.1}
Let $\mathcal{H}$ be a $(d-1)$--partite $(d-1)$--uniform hypergraph with $S\ge2$ edges. Suppose that $m,N,Z$ are positive integers with $N\ge S$ satisfying 
\[
m\geq 3, \quad N - S > Z >\max_{2 \le t \le S} \frac{\psi(N,t,m)}{t-1}, \quad \binom{m+ N - Z - S + 1}{m} > dS^2.
\]
Then, for every sufficiently large $n$, there is a $d$-partite $d$-uniform hypergraph with $\Theta(n)$ vertices and $\Omega(n^{d - 1/S})$ edges, which is $\mathcal{H}(t)$ free for any $t > m^{S + Z} \prod_{i=1}^{N- Z - S} M_i(d S^2)$.
\end{lemma}
\begin{proof}
We may delete isolated vertices from $\mathcal H$, since avoiding the
remaining hypergraph also avoids $\mathcal H$.
Put $r=N-Z-S$ and $\delta_i=M_{r-i+1}(dS^2)$ for $i\in[r]$.
Choose a sufficiently large prime $q$ with $q^S=\Theta_S(n)$.
Lemma~\ref{m-independent existence} supplies a pure
$(S+r)$-dimensional variety $\mathsf W$ of degree at most $m^Z$
whose rational points are $S$-wise $m$-independent and number at
least $q^{S+r}/2$.

Independently for $j\in[d]$, choose
$(h_{j,1},\dots,h_{j,r})$ uniformly from
$R_{\delta_1}\times\cdots\times R_{\delta_r}$ and put
$\mathsf X_j=\mathsf W\cap\mathsf V(h_{j,1},\dots,h_{j,r})$.
The regularity and point-counting lemmas imply, with probability
$1-o(1)$, that every $\mathsf X_j$ is pure of dimension $S$ and
\begin{equation}\label{3.1:eq4}
q^S/4<|\mathsf X_j(\mathbb F_q)|
\le 2m^Z\Bigl(\prod_{i=1}^r\delta_i\Bigr)q^S.
\end{equation}
Choose $g\in R_m^{\otimes d}$ independently and uniformly.
Declare a tuple $(x_1,\dots,x_d)$ to be an edge precisely when
$g(x_1,\dots,x_d)=0$.
For fixed vertex sets, distinct tuples give pairwise independent
edge indicators: distinct projective tuples give nonproportional
rank-one tensors in the tensor product of the Veronese spaces, so
their evaluation functionals are linearly independent in pairs.
Thus Chebyshev's inequality gives
$\Omega(q^{dS-1})$ edges with probability $1-o(1)$.

We verify freeness for all embeddings, not only for the prescribed
partition of $\mathcal H$.  In any embedding of $\mathcal H(t)$
into a $d$-partite hypergraph, one base edge already forces all apex
vertices into the same part.  Every nonisolated base vertex then
lies in one of the other $d-1$ parts.  Thus it suffices to consider
the finitely many assignments of base vertices to these parts that
make every base edge transversal.  Their number depends only on
$\mathcal H$ and $d$, and need not equal $d!$.

Fix such an assignment, with the apex in part $d$, and let its base
part sizes be $s_1,\dots,s_{d-1}$.  Each $s_i\le S$, since there
are no isolated vertices.  Condition on $\mathsf X_1,\dots,
\mathsf X_{d-1}$ satisfying \eqref{3.1:eq4}.  For each possible
base embedding $\phi$, Proposition~\ref{prop:lem-local independent}
shows that its $S$ partial evaluations of $g$ are independent uniform
elements of $R_m$.  Together with the independent slicing forms
for part $d$, they cut out the common apex variety
\[
\mathsf Y_\phi=\mathsf W\cap
\mathsf V\bigl(g(\phi(e),\mathord\cdot):e\in E(\mathcal H),
 h_{d,1},\dots,h_{d,r}\bigr).
\]
The parameter assumptions allow us to apply
Lemma~\ref{lem:regular sequence-probability-1} to these forms in the
order displayed, giving
$\Pr[\dim\mathsf Y_\phi>0]=O(q^{-dS^2})$.
There are $O_{\mathcal H,m,N,Z}(q^{S\sum_i s_i})$ possible
embeddings.  Since $\sum_i s_i\le(d-1)S$, a union bound gives
failure probability $O(q^{-S^2})$ for this assignment.
The same argument applies to every choice of the apex part and
base assignment.

Consequently, with probability $1-o(1)$, every common apex variety
is zero-dimensional.  B\'ezout's theorem bounds its number of points
by $m^{Z+S}\prod_i\delta_i<t$, so the constructed hypergraph is
$\mathcal H(t)$-free.  It has $\Theta(n)$ vertices and
$\Omega(n^{d-1/S})$ edges, as required.
\end{proof}
 
Now Theorem~\ref{thm-construction of apex partite} follows readily from Lemma~\ref{3.1} by choosing the parameters appropriately.

\begin{proof}[Proof of Theorem~\ref{thm-construction of apex partite}]
If $S=1$, take $d$ parts indexed by a cyclic group of order
$\lfloor n/d\rfloor$, with an edge when the sum of its coordinates
is zero.  Every $(d-1)$-set has at most one common neighbor, so this
hypergraph is $\mathcal H(t)$-free for $t\ge2$ and has
$\Omega_d(n^{d-1})$ edges.  We may therefore assume $S\ge2$.

Take $m=3$, $Z=S+r+3$ and $N=2S+2r+3$, where
$r=\lfloor(6dS^2)^{1/3}\rfloor$.
Lemma~\ref{m-independent quant} gives
$Z>\max_{2\le u\le S}\psi(N,u,3)/(u-1)$.
Moreover,
\[
\binom{m+N-Z-S+1}{m}=\binom{r+4}{3}>dS^2,
\qquad
3^{S+Z}\prod_{i=1}^rM_i(dS^2)
\le 3^{2S+r+3}(dS^2)^{1+\log r}r!<t.
\]
Thus Lemma~\ref{3.1} gives the required construction on
$\Theta(n)$ vertices.  To obtain exactly $n$ vertices, add isolated
vertices if necessary, or take an $n$-vertex subset chosen uniformly
at random.  In the latter case a constant proportion of the edges
survives in expectation.  This proves the stated lower bound.
\end{proof}

\section{Proof of Theorem~\ref{thm:sided-Zarankiewicz}}
\label{sec:sided-Zarankiewicz}

Let $d\ge2$ and $u, r, s_1,\dots,s_{d-1}$ be positive integers, and let
$\mathcal H$ be a $(d-1)$--partite $(d-1)$--uniform hypergraph
with ordered vertex parts of sizes $s_1,\dots,s_{d-1}$. We put
\[
S \coloneqq |E(\mathcal H)|,\qquad \delta_i \coloneqq M_{r-i+1}(2Su+1),\; i\in [r].
\]
We first prove the following lemma. 

\begin{lemma}\label{lem:sided-Zarankiewicz-parameterized}
Assume that $S \ge 2$. If $n_1,\dots, n_{d-1}, m, t$ are positive integers such that
\[
\binom{m+r+1}{m}>2Su+1,\qquad t > m^S \prod_{i=1}^r \delta_i,
\]
then for any sufficiently large $n_d$ such that $\log_{n_d}(\prod_{i=1}^{d-1}n_i^{s_i}) \le u$, we have
\[
\zz(n_1,\dots,n_d; \mathcal H(t)) = \Omega_{\mathcal H,u,r,m} \Big(
n_d^{1- \frac{1}{S}} \prod_{i=1}^{d-1}n_i \Big).
\]
\end{lemma}
\begin{proof}
If $n_i<s_i$ for some $i<d$, the complete partite hypergraph already
has the required property.  Assume otherwise, and put $N=S+r$ and
$\Delta=\prod_{i=1}^r\delta_i$.
Choose a prime $q$ with $4n_d\le q^S<2^{S+2}n_d$.
In particular, $q<8n_d^{1/S}$, independently of $m,r,u$.
Let $Y_i$ be a set of size $n_i$ for $i<d$.

For each tuple $\mathbf v\in Y_1\times\cdots\times Y_{d-1}$,
choose an independent uniform polynomial $P_{\mathbf v}\in R_m$.
Equivalently, these polynomials are the partial evaluations of the
multilinear form obtained by using coordinate points on the first
$d-1$ sides.  For a part-respecting base embedding $\phi$, let
$\mathsf Y_\phi=\mathsf V(P_{\phi(e)}:e\in E(\mathcal H))
\subseteq\mathbb P^N$.
There are at most $\prod_{i<d}n_i^{s_i}\le n_d^u\le q^{Su}$ such
embeddings.  The $S$ equations for each embedding are independent,
and $\binom{m+r+1}{m}>2Su+1$.  By
Lemma~\ref{lem:regular sequence-probability-1} and a union bound,
\begin{equation}\label{lem:sided-Zarankiewicz-parameterized:eq1}
\Pr[\text{every $\mathsf Y_\phi$ is pure of dimension $r$}]
=1-O_{S,r,u,m}(q^{-Su-1}).
\end{equation}
Every such variety has degree at most $m^S$.

Independently choose $g_i\in R_{\delta_i}$ uniformly and set
$\mathsf X=\mathsf V(g_1,\dots,g_r)$.
Conditioned on any choice of the $P_{\mathbf v}$ satisfying
\eqref{lem:sided-Zarankiewicz-parameterized:eq1},
Lemma~\ref{lem:general-simultaneous-slicing}, with
$\mathsf W=\mathbb P^N$, $b=Su$ and $E=m^S$, gives
\begin{equation}\label{lem:sided-Zarankiewicz-parameterized:eq2}
\dim\mathsf X=S,\qquad
|\mathsf X(\mathbb F_q)|\ge q^S/4\ge n_d,\qquad
\dim(\mathsf X\cap\mathsf Y_\phi)=0\quad\text{for every $\phi$}
\end{equation}
with probability $1-O_{S,r,u,m}(q^{-1})$.

For every choice of the $g_i$ with at least $n_d$ rational points,
select a set $X_0\subseteq\mathsf X(\mathbb F_q)$ of size $n_d$
by a fixed rule depending only on the $g_i$.
Condition on these $g_i$, but not on the regularity event involving
the $P_{\mathbf v}$.  The edge count
$Z_0=\sum_{\mathbf v\in Y_1\times\cdots\times Y_{d-1}}\sum_{x\in X_0}1_{P_{\mathbf v}(x)=0}$
has pairwise independent summands: different tuples use independent
polynomials, and distinct projective points impose independent
conditions on $R_m$.  Consequently, if
$\mu=n_d\prod_{i<d}n_i/q$, then
\begin{equation}\label{lem:sided-Zarankiewicz-parameterized:eq3}
\mathbb E[Z_0\mid(g_i)]=\mu,\qquad
\operatorname{Var}(Z_0\mid(g_i))\le\mu,\qquad
\Pr[Z_0<\mu/2\mid(g_i)]\le4/\mu=o(1).
\end{equation}
Here $S\ge2$ ensures that $\mu\to\infty$ uniformly in the allowed
part sizes.  Taking a union bound in the original product probability
space, \eqref{lem:sided-Zarankiewicz-parameterized:eq2} and
\eqref{lem:sided-Zarankiewicz-parameterized:eq3} hold together for
some choice of all polynomials.

Use $Y_1,\dots,Y_{d-1},X_0$ as the parts and declare
$(\mathbf v,x)$ to be an edge when $P_{\mathbf v}(x)=0$.
Every base embedding has at most $m^S\Delta<t$ common apex vertices,
by \eqref{lem:sided-Zarankiewicz-parameterized:eq2} and B\'ezout's
theorem.  The hypergraph therefore contains no ordered
$\mathcal H(t)$, and it has at least
\[
\frac{n_d\prod_{i<d}n_i}{2q}
\ge\frac1{16}n_d^{1-1/S}\prod_{i<d}n_i
\]
edges.  In particular, the multiplicative constant can be chosen
independently of $u,m,r$; only the threshold for $n_d$ depends on
these parameters.
\end{proof}

Theorem~\ref{thm:sided-Zarankiewicz} is a direct consequence of Lemma~\ref{lem:sided-Zarankiewicz-parameterized}.

\begin{proof}[Proof of Theorem~\ref{thm:sided-Zarankiewicz}]
For $S=1$, use a complete partite hypergraph with one active vertex
in the last part and pad that part with isolated vertices.  It has
$\prod_{i<d}n_i$ edges and no ordered $\mathcal H(t)$ for $t\ge2$.

For $S\ge2$, choose a minimizing pair $(m,\rho)$ in
\eqref{thm:sided-Zarankiewicz:threshold}.  The minimum is attained,
since the factorial term bounds $\rho$ in any bounded sublevel set,
and then $m$ is bounded as well.  Lemma~\ref{inequality} gives
$m^S\prod_{i=1}^\rho M_i(2Su+1)\le\kappa_S(u)<t$.
Apply Lemma~\ref{lem:sided-Zarankiewicz-parameterized} with $r=\rho$.
\end{proof}

\begin{proof}[Proof of Corollary~\ref{cor:sided-Zarankiewicz-complete}]
For the balanced case, take $u=\sum_{i<d}s_i\le(d-1)S$, $m=3$, and
$\rho=\lceil(6(2Su+1))^{1/3}\rceil$.
Then $\binom{m+\rho+1}{m}>2Su+1$, $\rho=O_d(S^{2/3})$, and
\[
\kappa_S(u)\le 3^S(2Su+1)^{1+\log\rho}\rho!
=3^{S+o_d(S)}.
\]
Theorem~\ref{thm:sided-Zarankiewicz} and the classical upper bound
prove part~\ref{cor:sided-Zarankiewicz-complete:item1}.

For the unbalanced case, fix $\rho=S$ and take
$m=\lceil((S+1)!(2Su+1))^{1/(S+1)}\rceil$.
This is admissible and gives
\[
\kappa_S(u)=O_S\bigl(u^{S/(S+1)+1+\log S}\bigr)
=O_S(u^{p_S}).
\]
Let $s_*=\max_{i<d}s_i$.
If $\log_{n_d}\prod_{i<d}n_i\le C s_d^{1/p_S}$, take
$u=\lceil s_*C s_d^{1/p_S}\rceil$.
For a sufficiently small $C=C(s_1,\dots,s_{d-1})$ and sufficiently
large $s_d$, we have $\kappa_S(u)<s_d$, while
$\prod_{i<d}n_i^{s_i}\le n_d^u$.
Theorem~\ref{thm:sided-Zarankiewicz} proves the lower bound.
The upper bound under $n_d\le\min_{i<d}n_i$ follows from
\cite[Corollary~1.2]{GaoHouHuangWang2025}.
\end{proof}

\section{Proof of Theorem~\ref{thm:generalized Turan hypergraph explicit}}\label{sec:generalized-Turan}

\begin{lemma}[Hyperedge Probability]\label{lem:hyperedge probability}
Let $\mathcal U\subseteq S^d(R_m)$ be an $\alpha$-dimensional subspace. Suppose that $E \subseteq \binom{\mathbb P^N(\mathbb F_q)}{d}$ with $|E| \le \alpha$. If $F$ is uniformly chosen from $\mathcal U\setminus\{0\}$ at random, then
\[
\Pr[F(v_1,\dots,v_d)=0 \text{~for every~} \{v_1,\dots,v_d\} \in E]
\ge\frac{q^{\alpha-|E|}-1}{q^\alpha-1}.
\]
\end{lemma}
\begin{proof}
Denote $M \coloneqq \dim R_m = \binom{N + m}{m}$. Let $f_1,\dots ,f_M$ be a basis of $R_m$. Then we may expand each $G \in \mathcal{U} \subseteq S^d R_m$ as a linear combination of $f_{i_1} \circ \cdots \circ f_{i_d}$, where $1 \le i_1 \le \cdots \le i_d \le M$. For a fixed $\{v_1, \dots, v_d\}\in E$, $G(v_1, \dots, v_d) = 0$ supplies a linear equation on $\mathcal{U}$. Set $\mathcal{V} \coloneqq \{ G \in \mathcal{U}: G(v_1,\dots, v_d) = 0, \{v_1,\dots, v_d\} \in E \}$. Then $\mathcal{V}$ is a linear subspace of $\mathcal{U}$ of codimension at most $|E|$, from which we obtain
\[
\Pr\!\big[ F(x_1,\dots, x_d) = 0, \{x_1,\dots, x_d\}\in E \big] = \frac{|\mathcal{V} \setminus \{0\}|}{|\mathcal{U}\setminus \{0\}|} \ge \frac{q^{\alpha - |E|} - 1}{q^\alpha - 1}. \qedhere
\]   
\end{proof}

\begin{lemma}\label{lem:generalized Turan hypergraph parameterized}
Suppose $d\ge2$. Let $\mathcal{F}$ be a $d$--uniform hypergraph with $v$ vertices
and $e$ edges, and let $\mathcal{B}$ be a $(d-1)$--uniform
hypergraph with $\sigma$ vertices and $S$ edges. For any positive integers $m,N,Z,t$ satisfying
\begin{align*}
m&\ge3,\qquad N \ge \sigma,\qquad N-Z \ge S,\qquad
 Z>\max_{2\le u\le\sigma}\frac{\psi(N,u,m)}{u-1},\\
e &\le
\binom{m+N-Z-S+1}{m}-(N-Z)\sigma- \min\{S-1,\sigma-d+1\} -2, \\
t &> m^{Z+S}\prod_{i=1}^{N-Z-S}M_{N-Z-S-i+1}( (N-Z)\sigma+e+2),
\end{align*}
then we have $\ex (n,\mathcal F,\mathcal B(t))
=\Omega_{\mathcal F,\mathcal B}
(n^{v-e/S})$.
\end{lemma}
\begin{proof}
For convenience, we denote
\[
k \coloneqq N-Z,\qquad \alpha \coloneqq e+1,\qquad \rho\coloneqq k -S,
\qquad
\delta_i\coloneqq M_{\rho-i+1}(k\sigma+e+2)\;, i\in [\rho].
\]
According to Bertrand's postulate, for any sufficiently large $n$, there exists a prime $q$ such that $q>d$ and $4n\le q^S<2^{S+2}n$. By Proposition~\ref{prop:symmetric choice of parameters}, we may further increase $n$ if necessary, so that there exist $f_1,\dots,f_Z\in R_m$, $g_i \in R_{\delta_i}$, $i \in [\rho]$, and an $\alpha$--dimensional subspace $\mathcal{U} \subseteq S^d R_m$, satisfying \ref{prop:symmetric choice of parameters:item1}--\ref{prop:symmetric choice of parameters:item3} in Proposition~\ref{prop:symmetric choice of parameters}. Denote $\mathsf{W} \coloneqq \mathsf{V}(f_1,\dots, f_Z)$ and $\mathsf{X} \coloneqq \mathsf{W} \cap \mathsf{V}(g_1,\dots, g_{\rho})$. 

We pick a subset $X_0\subseteq\mathsf X(\mathbb F_q)$ with $|X_0| = n$. For each $F\in\mathcal U\setminus\{0\}$, we consider the $d$--uniform hypergraph
$\mathcal{G}_F$ with 
\[
V(\mathcal{G}_F) = X_0,\qquad E(\mathcal{G}_F) =\Big\{ \{v_1,\dots, v_d\} \in \binom{X_0}{d}: F(v_1,\dots, v_d) = 0 \Big\}.
\]
We claim that $\mathcal{G}_F$ does not contain any copy of $\mathcal{B}(t)$. Otherwise, there exist $v_1,\dots, v_\sigma \in X_0$ and $\mathcal{I} \subseteq \binom{[\sigma]}{d-1}$ with $|\mathcal{I}| = S$ such that $V(\mathcal{B}) = \{v_1,\dots, v_{\sigma}\}$ and $E(\mathcal{B}) = \mathcal{I}$. Moreover, all the $t$ apex vertices of $\mathcal{B}(t)$ must be contained in $X_0 \cap \mathsf{V}(F_I: I\in \mathcal{I})$. However, by Proposition~\ref{prop:symmetric choice of parameters}--\ref{prop:symmetric choice of parameters:item3}, we obtain a contradiction:
\[
t\le |X_0 \cap \mathsf{V}(F_I: I\in \mathcal{I})| \le | (\mathsf{X}\cap \mathsf{V}(F_I: I\in \mathcal{I}))(\mathbb{F}_q)| \le \deg (\mathsf{X}\cap \mathsf{V}(F_I: I\in \mathcal{I})) \le m^{Z+S}\prod_{i=1}^{\rho}\delta_i < t.
\]

For each fixed injection $\varphi:V(\mathcal F)\to X_0$, the images
of its $e$ edges impose at most $e$ linear conditions on $\mathcal U$.
Choose $F$ uniformly from $\mathcal U\setminus\{0\}$.
Lemma~\ref{lem:hyperedge probability}, with $\alpha=e+1$, gives
\[
\Pr_F[\varphi\text{ embeds }\mathcal F\text{ into }\mathcal G_F]
\ge\frac{q-1}{q^{e+1}-1}.
\]
Summing over injections and dividing by the number of automorphisms,
we obtain
\[
\mathbb E_F\,\ccc(\mathcal F,\mathcal G_F)
\ge \frac{(n)_v}{|\operatorname{Aut}(\mathcal F)|}
\frac{q-1}{q^{e+1}-1}
=\Omega_{\mathcal F,\mathcal B}(n^{v-e/S}).
\]
Every $\mathcal G_F$ is $\mathcal B(t)$-free, so some $F$ has at least
this many copies.
\end{proof}

Theorem~\ref{thm:generalized Turan hypergraph explicit} follows immediately from Lemma~\ref{lem:generalized Turan hypergraph parameterized}.

\begin{proof}[Proof of Theorem~\ref{thm:generalized Turan hypergraph explicit}]
We set $m \coloneqq 3$, $Z \coloneqq S+\rho+\sigma+d+3$ and $N \coloneqq S+\rho+Z$. It is straightforward to verify that
\[
Z > \max_{2\le u\le\sigma}\frac{\psi(N,u,3)}{u-1}.
\]
Note that \[
\binom{4+N-Z-S}{3}
=\binom{\rho+4}{3}>\frac{7S\sigma}{6}.
\]
Thus, we have
\[
e \le \binom{N-Z-S+4}{3}-(N-Z)\sigma- \min\{S-1,\sigma-d+1\} -2.
\]
Moreover, Lemma~\ref{inequality} yields that
\[
t > 3^{2S+\rho+\sigma+d+3}T^{1+\log\rho}\rho! \ge 3^{Z+S}\prod_{i=1}^{\rho}M_i(T).
\]
We obtain $\ex(n,\mathcal{F},\mathcal{B}(t)) = \Omega_{\mathcal{F},\mathcal{B}} (n^{v-e/S})$ by Lemma~\ref{lem:generalized Turan hypergraph parameterized}.
\end{proof}

\section{Proof of Theorem~\ref{thm-ramsey_1}}\label{sec:algebraic-coloring}

\begin{lemma}\label{lem-Ramsey_1-construction}
Let $\mathcal F$ be a $(d-1)$--uniform hypergraph with $\sigma$ vertices and $S$ edges. Suppose that
$N,m,Z$ are positive integers satisfying
\begin{align*}
&N \ge \sigma,\qquad
N>Z>\max_{2\le u\le \sigma} \frac{\psi(N,u,m)}{u-1},\qquad N-Z \ge S, \\
&\binom{m+N-Z-S+1}{m}-(N-Z)\sigma-\min\{S-1,\sigma-d+1\} \ge 3.
\end{align*}
If $t>m^{Z+S}\prod_{i=1}^{N- Z - S}M_i((N-Z)\sigma + 3)$, then for each sufficiently large $n$, there is an edge-coloring of
$\mathcal{K}_n^{(d)}$ with $\Theta_{\mathcal F}(n^{1/S})$ colors and no
monochromatic copy of $\mathcal F(t)$.
\end{lemma}
\begin{proof}
Denote
\[
k\coloneqq N-Z,
\qquad
\rho\coloneqq k-S,
\qquad
T\coloneqq k\sigma+3,
\qquad
\delta_i\coloneqq M_{\rho-i+1}(T)\quad(1\le i\le\rho).
\]
By Bertrand's postulate, there is a prime $q>d$ such that $4n\le q^S <2^{S+2}n$ for any sufficiently large $n$. In particular, we have $q = \Theta_S(n^{1/S})$. We may further increase $n$ so that Proposition~\ref{prop:symmetric choice of parameters} is applicable. Thus, there exist subvarieties $\mathsf{W},\mathsf{X} \subseteq \mathbb{P}^N$, a two-dimensional subspace $\mathcal {U} \subseteq S^d R_m$, and polynomials $g_1 \in R_{\delta_1},\dots,g_\rho \in R_{\delta_\rho}$ satisfying \ref{prop:symmetric choice of parameters:item1}--\ref{prop:symmetric choice of parameters:item3} in Proposition~\ref{prop:symmetric choice of parameters}.

Let $X_0 \subseteq \mathsf{X}(\mathbb{F}_q)$ be a subset consisting of $n$ points. We consider the complete $d$-uniform hypergraph $\mathcal{K}_n^{(d)}$ with the vertex set $X_0$. We color edges of $\mathcal{K}_n^{(d)}$ by points of $\mathbb P(U)(\mathbb F_q)$ as follows. Since $\dim \mathcal{U} = 2$, $\mathbb{P}(\mathcal{U})$ is a projective line, which has $q+1$ points over $\mathbb{F}_q$. We choose and fix an ordering of these points. For each $e = \{v_1,\dots, v_d\} \subseteq X_0$, we define the evaluation map
\[
\operatorname{ev}_e:\mathcal{U} \to \mathbb{F}_q, \quad \operatorname{ev}_e(F) \coloneqq F(v_1,\dots,v_d).
\]
It is clear that $\dim \ker(\operatorname{ev}_e) \ge 1$. We color $e$ by the first element in $\mathbb{P}(\ker(\operatorname{ev}_e) ) \subseteq \mathbb{P} (\mathcal{U})$ with respect to the chosen ordering on $\mathbb{P} (\mathcal{U})$.

We claim that $\mathcal{K}_n^{(d)}$ contains no monochromatic copy of $\mathcal{F}(t)$. Otherwise, there exist $F \in \mathcal{U} \setminus \{0\}$, $V \subseteq X_0$ with $|V| = \sigma$ and $\mathcal{I} \subseteq \binom{V}{d-1}$ with $|\mathcal{I}| = S$, such that the hypergraph with vertex set $V$ and edge set $\mathcal{I}$ is a copy of $\mathcal{F}$, and $X_0 \cap \mathsf{V}(F_I: I\in \mathcal{I})$ contains at least $t$ points. By construction, however, $\mathsf{X} \cap \mathsf{V}(F_I: I\in \mathcal{I})$ is zero-dimensional, which yields a contradiction:  
\[
|X_0 \cap \mathsf{V}(F_I: I\in \mathcal{I})| \le |(\mathsf{X} \cap \mathsf{V}(F_I: I\in \mathcal{I}))(\mathbb{F}_q)| \le m^{Z + S} \prod_{i=1}^\rho \delta_i < t. \qedhere
\]
\end{proof}

An appropriate choice of parameters in Lemma~\ref{lem-Ramsey_1-construction} yields a proof of Theorem~\ref{thm-ramsey_1}.

\begin{proof}[Proof of Theorem~\ref{thm-ramsey_1}]
We set
\[
\tau \coloneqq \min\{S-1,\sigma-d+1\}, \qquad
k \coloneqq S+\rho,
\qquad
Z\coloneqq S+\rho+\sigma+d+3,
\qquad
N\coloneqq S+\rho+Z.
\]
Then it is straightforward to verify that 
\[
Z>\max_{2\le u\le\sigma} \psi(N,u,3)/(u-1),\qquad \binom{4+N-Z-S}{3} >\frac{7S\sigma}{6} > k\sigma + \tau + 3.
\]
Indeed, the absence of isolated vertices gives
$\sigma\le(d-1)S$ and hence $\rho=O_d(S^{2/3})$.
Since $\tau\le\sigma$, we have
$k\sigma+\tau+3\le(S+O_d(S^{2/3})+1)\sigma+3
<7S\sigma/6$ for all sufficiently large $S$ depending only on $d$. By Lemma~\ref{inequality}, we have 
\[
t > 3^{2S+\rho+\sigma+d+3} (k \sigma + 3)^{1+\log\rho}\rho! \ge 3^{Z+S}\prod_{i=1}^{\rho}M_i(k \sigma + 3).
\]
Lemma~\ref{lem-Ramsey_1-construction} with $m=3$ implies that there is a $\Theta_{\mathcal{F}}(n^{1/S})$-coloring of $\mathcal{K}^{(d)}_n$ with no monochromatic copy $\mathcal{F}(t)$. Taking $n = \Theta(K^S)$ with small constant yields $\rrr_K^{(d)}(\mathcal{F}(t))=\Omega_{\mathcal{F},t}(K^S)$.
\end{proof}

\section{Proof of Theorem~\ref{thm:Ramsey}}
\label{sec:pseudorandom-Ramsey}

We first establish Lemmas~\ref{edge-independence}--\ref{probability of K_s,t-free} using the algebraic multilinear framework developed in Section~\ref{sec:multilinear-framework}.

\begin{lemma}\label{edge-independence}
Suppose that $\ch \mathbb{F}_q > d$. Let $X \subseteq \mathbb P^N(\mathbb F_q)$ be an $\ell$--wise $m$--independent set with $|X|\ge\ell$. Assume that $\mathcal{E} \subseteq \binom{X}{d}$ satisfies $|\cup_{e \in \mathcal{E}} e| \le \ell$. If $F \in S^d R_m$ is chosen uniformly, then the random variables $1_{F(e)=0}$ where $e\in\mathcal{E}$ are mutually independent, and each has expectation $1/q$.  
\end{lemma}
\begin{proof}
We write $Y \coloneqq \cup_{e \in \mathcal{E}} e = \{ [v_1],\dots, [v_a] \}$ where $a \le \ell$. Since $X$ is $\ell$-wise $m$-independent, $Y$ is $m$-independent. By Lemma~\ref{characterization of m-independence}, there exist $f_1,\dots, f_a\in R_m$ such that $f_i(v_j)=\delta_{ij}$, where $i,j \in[a]$. Given $e = \{[v_{i_1}],\dots, [v_{i_d}]\} \in \mathcal{E}$, we consider the polynomial 
\[
F_e \coloneqq d! (f_{i_1} \circ \cdots \circ f_{i_d}) \in S^d R_m.
\]
By the choice of $f_i$, we have $F_e(e') = \delta_{e,e'}$. In particular, the linear map
\[
S^d R_m \to \mathbb{F}_q^{|\mathcal{E}|},\qquad F \mapsto (F(e))_{e \in \mathcal{E}}
\]
is surjective. Therefore, for a uniformly random $F \in S^d R_m$, $(F(e))_{e\in \mathcal{E}}$ is uniform on $\mathbb{F}_q^{|\mathcal{E}|}$. As a consequence, the random variables $F(e)$, where $e\in \mathcal{E}$, are mutually independent and uniform. This implies that the indicator functions $1_{F(e) = 0}$ are mutually independent and each has expectation $1/q$.
\end{proof}

\begin{lemma}\label{g exp}
Suppose that $q > a$, $\delta \ge m$, and $\{[v_1],\dots,[v_a]\} \subseteq \mathbb{P}^N(\mathbb{F}_q)$ is $m$--independent. If $g$ is chosen uniformly from $R_\delta$, then we have $\mathbb{E} \big[\prod_{j=1}^{a}1_{g(v_j)=0} \big]=q^{-a}.$
\end{lemma}
\begin{proof}
By Lemma~\ref{characterization of m-independence}, there exist
$f_1,\dots, f_a\in R_m$ such that $f_i(v_j)=\delta_{ij}$, where $i,j\in [a]$. Since $q > a$, there is some $L \in R_1$ such that $L(v_j) \ne 0$ for each $j \in [a]$.

We consider
\[
h_i \coloneqq L(v_i)^{-(\delta-m)} L^{\delta-m} f_i \in R_\delta,\; i\in [a].
\]
Since $h_i(v_j) = \delta_{ij}$ for $i,j\in [a]$, the linear map
\[
R_\delta \to \mathbb{F}_q^a,\qquad g \mapsto (g(v_j))_{j\in[a]}
\]
is surjective. Thus, for a uniformly random $g \in R_\delta$, a similar argument as in the proof of Lemma~\ref{edge-independence} shows that the indicator functions $1_{g(v_j) = 0}$ where $j \in [a]$, are mutually independent and each has expectation $1/q$. This implies that $\mathbb{E}\!\big[\prod_{j=1}^a 1_{g(v_j)=0} \big]
= \prod_{j=1}^a \mathbb{E}[1_{g(v_j)=0}] =q^{-a}$.
\end{proof}

Let $d,N,m,Z$ be positive integers, and let $\mathcal F$ be a
$(d-1)$--uniform hypergraph. We denote 
\begin{align*}
S &\coloneqq |E(\mathcal F)|,
\qquad
\sigma\coloneqq |V(\mathcal F)|,
\qquad
\tau\coloneqq\min\{S-1,\sigma-d+1\}, \\
k &\coloneqq N-Z,\qquad
\rho\coloneqq k-S,\qquad
T\coloneqq k\sigma+3,\qquad
\delta_i\coloneqq M_{\rho-i+1}(T), \;1\le i\le\rho. 
\end{align*}
Suppose that these numbers satisfy the following conditions:
\begin{align*}
S &\le\binom{\sigma}{d-1},\qquad N \ge \sigma,\qquad
N>Z>\max_{2\le u\le \sigma} \frac{\psi(N,u,m)}{u-1},\qquad N-Z \ge S, \\
3 &\le
\binom{m+N-Z-S+1}{m}-(N-Z)\sigma- \tau.
\end{align*}
The displayed assumptions allow us to apply
Lemma~\ref{m-independent existence}.  Let $q$ be a sufficiently large
prime power with $\ch \mathbb{F}_q>d$, and let
$\mathsf W=\mathsf V(f_1,\dots,f_Z)\subseteq\mathbb P^N$ be the resulting
pure $k$--dimensional, $\sigma$--wise $m$--independent variety.  In
particular,
\[
|\mathsf W(\mathbb F_q)|\ge\frac12q^k,
\qquad
\deg\mathsf W\le D \coloneqq m^Z .
\]
Next, we uniformly and independently choose
\[
(F, g_1, \dots, g_{\rho}) \in S^d R_m \times R_{\delta_1} \times \cdots \times R_{\delta_\rho}.
\]
We write $\mathsf X\coloneqq
\mathsf W\cap\mathsf V(g_1,\dots,g_\rho)$, and define a random $d$--uniform hypergraph $\mathcal{G}$ with
\[
V(\mathcal{G}) = \mathsf{X}(\mathbb F_q),\qquad E(\mathcal{G}) = \{\{v_1,\dots, v_d\} \subseteq \mathsf{X}(\mathbb F_q): F(v_1,\dots, v_d) = 0 \}.
\]  

\begin{lemma}\label{probability of K_s,t-free}
Let $\mathcal{G}$ be a random $d$-uniform hypergraph defined as above. For any $t>m^{Z+S}\prod_{i=1}^{\rho}\delta_i$, the following holds with probability at least $2/3$:
\begin{enumerate}[\normalfont(i)]
\item $\mathcal{G}$ does not contain 
$\mathcal F(t)$. \label{probability of K_s,t-free:item1}
\item $|V(\mathcal G)|=\Theta(q^S)$. \label{probability of K_s,t-free:item2}
\end{enumerate}
\end{lemma}
\begin{proof}
We identify $V(\mathcal F)$ with $[\sigma]$. Let $\mathfrak C$ be the
family of injective maps $\phi:V(\mathcal F)\to\mathsf W(\mathbb F_q)$.
Note that $|\mathsf{W}(\mathbb F_q)|=O(q^k)$. Thus, we must have 
\[
|\mathfrak C|\le |\mathsf W(\mathbb F_q)|^\sigma=O(q^{k\sigma}).
\]
For each $\phi\in\mathfrak C$, write $\phi(i)=[v_i]$ for $i\in[\sigma]$
and define
\[
\mathcal I_\phi\coloneqq E(\mathcal F)\subseteq\binom{[\sigma]}{d-1}.
\]
Then $|\mathcal I_\phi|=S$. We consider
\[
\mathsf Y_F(\phi) \coloneqq
\mathsf{W} \cap \mathsf{V}(
F(v_{i_1},\dots,v_{i_{d-1}},\cdot):
\{i_1,\dots,i_{d-1}\}\in E(\mathcal F)).
\]

For a fixed $\phi\in\mathfrak C$, Proposition~\ref{prop:symmetric partial evaluations} with $\mathcal{I} = \mathcal I_\phi$ yields 
\[
\Pr \bigl[\dim\mathsf Y_F(\phi) > \rho\bigr] = O( q^{\tau-\binom{m+\rho+1}{m}}) = O(q^{- k\sigma - 3}),
\]
where the last inequality follows from the assumption that $3 \le \binom{m + \rho + 1}{m} - k\sigma - \tau$. It is clear that $\deg Y_F(\phi) \le m^{Z + S}$. On the other side, every irreducible component of $\mathsf{Y}_F(\phi)$ has dimension at least $k-S=\rho$, so if $\dim \mathsf{Y}_F(\phi)\le\rho$ occurs, $\mathsf{Y}_F(\phi)$ is pure of dimension $\rho$. As a consequence, we obtain that
\[
\Pr \bigl[\mathsf Y_F(\phi) \text{~is pure of dimension $\rho$ for each~} \phi\in\mathfrak C \bigr] = 1 - O(q^{-3}).
\]

Next, let $F$ be such that $Y_F(\phi)$ is pure of dimension $\rho$ for each $\phi\in\mathfrak C$. We denote by $\mathfrak {Y}_F$ the family of all $\mathsf{Y}_F(\phi)$, where $\phi\in\mathfrak C$. Then $|\mathfrak {Y}_F| \le |\mathfrak{C}| = O(q^{k \sigma})$. We observe that
\[
\delta_i = M_{\rho-i+1}(k\sigma+3)
\ge M_{\rho-i+1}(k\sigma+1)
\]
for each $i \in [\rho]$. Taking $b = k\sigma$, $D = m^Z$ and $E = m^{Z+S}$ in Lemma~\ref{lem:general-simultaneous-slicing}, we conclude that, with probability $1 - O(q^{-1})$ over the choice of $(g_1,\dots, g_{\rho})$, $\mathsf X$ is pure of dimension $S$ with $|\mathsf{X}(\mathbb F_q)|\ge q^S/4$ and $\dim(\mathsf X\cap\mathsf{Y})=0$ for every $\mathsf{Y} \in \mathfrak{Y}_F$. We claim that in this case, the corresponding $\mathcal{G}$ does not contain a copy of $\mathcal F(t)$. Otherwise, the base of such a copy supplies some $\phi\in\mathfrak C$. Moreover, all $t$ apex vertices of $\mathcal F(t)$ are contained in $(\mathsf{X} \cap \mathsf{Y}_F(\phi))(\mathbb{F}_q)$, which contradicts
\[
t > m^{Z+S}\prod_{i=1}^{\rho}\delta_i \ge \deg (\mathsf{X} \cap \mathsf{Y}_F(\phi)),
\]
since $\mathsf{X} \cap \mathsf{Y}_F(\phi)$ is zero-dimensional. Together with Lemma~\ref{generalized schwarz-zippel} and $\deg\mathsf{X} \le m^Z\prod_{i=1}^{\rho}\delta_i$, the construction of $\mathcal{G}$ implies that 
\[
O(q^S) = |V(\mathcal{G})| = |\mathsf{X}(\mathbb{F}_q)| \ge \frac{q^S}{4},
\]
from which we obtain that $|V(\mathcal{G})| = \Theta(q^S)$.

Lastly, the above calculations show that \ref{probability of K_s,t-free:item1} and \ref{probability of K_s,t-free:item2} hold with probability $1 - O(q^{-1})$, which tends to one when $q$ goes to infinity. Therefore, they hold with probability at least $2/3$ for any sufficiently large $q$.
\end{proof}

\subsection{Pseudorandom \texorpdfstring{$\mathcal{K}_{s,t}$}--free graphs}
\label{subsec:pseudorandom-graph}

In this subsection, we focus on the existence of $\mathcal{K}_{s,t}$--free graphs with sufficiently many vertices and large subgraphs. Let $s$ be a sufficiently large positive integer. Suppose that $\mathcal{G}$ is the random graph defined as in Lemma~\ref{probability of K_s,t-free} with $|V(\mathcal{F})| = |E(\mathcal{F})| = s$ and
\begin{equation}\label{subsec:pseudorandom-graph:parameters}
d = 2,\qquad m = 3,\qquad \rho =  \lfloor(7s^2)^{1/3}\rfloor, \qquad Z = s+ \rho + 5, \qquad N = 2(s + \rho) + 5.
\end{equation}
Here $k=s+\rho$, $T=ks+3$, and $D=3^Z$.
For sufficiently large $s$, Lemma~\ref{m-independent quant} verifies
the required independence condition, and
$\binom{\rho+4}{3}>7s^2/6>ks+s+2$ verifies the dimension-drop
condition.  Also $\binom{\rho+2}{2}<T$, so
$\delta_i\ge M_\rho(T)\ge3=m$ for every $i$.
This last inequality is needed for the vertex-survival calculation
in Appendix~\ref{appendix:spectral}.

Given a subset $U \subseteq V(\mathcal{G})$, we denote by $E_U(\mathcal{G})$ the edge set of the subgraph of $\mathcal{G}$ with the vertex set $U$. The following technical lemma is essential for our argument. Its proof is deferred to Appendix~\ref{appendix:spectral}.

\begin{lemma}[Local Density]\label{lem:local-density}
Let $h\le s/2$ be a positive integer and put
\[
R_h=128(h+1)Dq^{(s+1)/2+s/(2h)}.
\]
With probability at least $1-2^{-4h}$, \emph{every}
$U\subseteq V(\mathcal G)$ with $|U|\ge R_h$ satisfies
$|E_U(\mathcal G)|\ge |U|^2/(8q)$.
\end{lemma}

\begin{proposition}[Pseudorandom $\mathcal{K}_{s,t}$--free Graph]\label{prop:random-exists}
The following properties of
$\mathcal{G}$ hold simultaneously with probability at least $1/3$:
\begin{enumerate}[\normalfont(i)]
\item There is an absolute constant $C$  such that $\mathcal G$ is $\mathcal K_{s,t}$--free for every $t>C9^s s^{4s^{2/3}}$. \label{prop:random-exists:item1}
\item $|V(\mathcal G)| = \Theta(q^s)$.  \label{prop:random-exists:item2}
\item For every positive integer $h \le s/2$ and $U\subseteq V(\mathcal G)$ with $|U| \ge 128(h+1)D q^{\frac{s+1}{2}+\frac{s}{2h}}$, we have $|E_U(\mathcal G)| \ge\frac{|U|^2}{8q}$.  \label{prop:random-exists:item3}
\end{enumerate}
\end{proposition}
\begin{proof}
The degree threshold in Lemma~\ref{probability of K_s,t-free} is at
most $9^s B_s$, where
$B_s=3^{\rho+5}T^{1+\log\rho}\rho!$.
Since $\rho=7^{1/3}s^{2/3}+O(1)$ and $T=\Theta(s^2)$,
\[
\log B_s
\le\rho\log(3\rho)+O((\log s)^2)
=\left(\frac23\,7^{1/3}+o(1)\right)s^{2/3}\log s
<4s^{2/3}\log s
\]
for sufficiently large $s$.  Thus the required threshold is bounded
by $C9^s s^{4s^{2/3}}$ with an absolute $C$.
Lemma~\ref{probability of K_s,t-free} gives
\ref{prop:random-exists:item1} and \ref{prop:random-exists:item2}
together with probability at least $2/3$.
By the simultaneous formulation of Lemma~\ref{lem:local-density},
a union bound over $1\le h\le\lfloor s/2\rfloor$ gives
\ref{prop:random-exists:item3} with probability at least
$1-\sum_{h\ge1}2^{-4h}=14/15$.
The three conclusions therefore hold together with probability
at least $2/3-1/15>1/3$.
\end{proof}

\subsection{Proof of Theorem~\ref{thm:Ramsey}}

This subsection is devoted to proving Theorem~\ref{thm:Ramsey}. For ease of reference, we record the following three Lemmas. 

\begin{lemma}\label{indep-set}\cite[Lemma 3.1]{alon2005sharp}
Let $\mathcal{G}$ be a graph with $n$ vertices, and let $M$ be the number of its
independent sets of size $m$. If $K \ge2$ and $M^K<\binom{n}{m}^{K-1}$, then there are graphs $\mathcal{G}_1,\dots,\mathcal{G}_K$ on a common $n$--vertex set,
each isomorphic to $\mathcal{G}$, such that the graph consisting of the pairs that
belong to none of the $\mathcal{G}_i$ contains no $\mathcal{K}_m$. In particular, if $\mathcal{G}$ is
$\mathcal{H}$--free, then $\rrr_K(\mathcal{H};\mathcal{K}_m) > n$.
\end{lemma}

\begin{lemma}\label{container}\cite[Lemma 3.1]{Kohayakawa2015}
Let $\mathcal{G}$ be a graph with $n$ vertices. Suppose that $f$ is a positive integer, and $R,\beta$ are positive real numbers with $\beta < 1$ and $e^{-\beta f}n<R$. If every $U\subseteq V(G)$ with $|U|\ge R$ satisfies $|E_U(\mathcal{G})| \ge \beta|U|^2/2$, then for every integer $m\ge f$, the number of independent sets of size $m$ in $\mathcal{G}$ is at most $\binom nf\binom R{m-f}$.
\end{lemma}

\begin{lemma}\label{upper-bound}\cite[Lemma~3.3]{alon2005sharp}
Let $K$ be a positive integer, and let $\mathcal{H}$ be a bipartite graph with $\ex(n,\mathcal{H}) = O(n^{2-1/\theta})$ for some real number $\theta > 1$. Then there is a constant $C \coloneqq C(\mathcal{H},K)$ such
that $\rrr_K(\mathcal{H}; \mathcal{K}_m)
\le Cm^\theta/\log^\theta m$.
\end{lemma}

Now we are ready to prove Theorem~\ref{thm:Ramsey}.

\begin{proof}[Proof of Theorem~\ref{thm:Ramsey}]
The upper bound follows from Lemma~\ref{upper-bound} with $\mathcal{H} = \mathcal{K}_{s,t}$ and $\theta=s$. According to \cite[Theorem~3.7]{alon2005sharp}, we increase $C$ if necessary, so that we may assume that $s$ is sufficiently large.     

We first take $L \coloneqq \lfloor 14s q\log q\rfloor$. By Proposition~\ref{prop:random-exists} with $h=4$, there is a $\mathcal{K}_{s,t}$--free graph $\mathcal{G}$
such that
\[
M\coloneqq|V(G)|=\Theta(q^s)
=\Theta \Big(\frac{L^s}{\log^s L}\Big),
\]
and every subset $U \subseteq V(G)$ satisfies $|E_U(\mathcal{G})| \ge |U|^2/(8q)$ whenever 
\[
|U| \ge R\coloneqq
\left\lceil 640 Dq^{\frac{5s+4}{8}}\right\rceil
=\Theta(q^{s-\eta}),\quad \eta \coloneqq \frac{3s - 4}{8}.
\]

Next, we write $I_L(\mathcal{G})$ for the number of independent $L$--subsets of $\mathcal{G}$, and take $f_0\coloneqq
\left\lceil 4q \log\frac{M}{R}\right\rceil+1$. Since $\log(M/R)=\eta\log q+O(1)$, we have $f_0= (4\eta + o(1))q\log q$. In particular, $f_0<L$ for all sufficiently large $q$, and $e^{-f_0 / (4q)}M<R$. Thus, Lemma~\ref{container} implies that $I_L(\mathcal{G})\le\binom M{f_0}\binom R{L-f_0}$. By the standard estimates for binomial coefficients, we have
\[
\log I_L(\mathcal{G}) \le (A +o(1))q\log^2q,\qquad 
\log\binom ML \ge (B+o(1))q\log^2q,
\]
where $A \coloneqq (149s^2-136s+16)/16$ and $B \coloneqq 14s(s-1)$. Here $2B-3A=(s^2-40s-48)/16>0$ for $s\ge42$.
Since $B>A$, this implies $(K-1)B>KA$ for every $K\ge3$. Therefore, Lemma~\ref{indep-set} yields $\rrr_K( \mathcal{K}_{s,t}; \mathcal{K}_L)
=\Omega(L^s/\log^s L)$.

Finally, by Bertrand's postulate, for every sufficiently large
integer $n$, there is an odd prime $q$ such that $q = \Theta(n/\log n)$ and $L\le n = O(L)$. Thus, we have 
\[
\rrr_K(\mathcal{K}_{s,t}; \mathcal{K}_n) \ge \rrr_K(\mathcal{K}_{s,t}; \mathcal{K}_L) 
=\Omega(L^s/\log^s L) = \Omega(n^s/\log^s n). \qedhere
\]
\end{proof}

\section{Concluding remarks}\label{sec:concluding}
Our constructions connect a local geometric condition---uniform
control of partial evaluations---to extremal density, prescribed
subgraph counts, and Ramsey colorings.  Three questions appear
particularly important for extending this connection.

\paragraph{Beyond Sidorenko bases.}
The main unresolved issue is the gap between $|E(\mathcal H)|$,
which governs our lower bound, and $\sss(\mathcal H)$, which governs
Lee's upper bound.  For a non-Sidorenko base these differ, so the
present construction does not attain the conjectured exponent.
A concrete test case is the $3$-uniform loose triangle $\mathcal L$
with three edges, for which $\sss(\mathcal L)=4$
\cite{lee2025sidorenko}.  For sufficiently large fixed $t$, our lower
bound and Lee's upper bound give
\[
\Omega_t(n^{11/3})\le\ex(n,\mathcal L(t))\le O_t(n^{15/4}).
\]
Determining whether the upper exponent is attained would test
Conjecture~\ref{conj:Lee} beyond the Sidorenko regime.
A construction at that exponent would need to encode more of the
base's homomorphism structure than one independent equation per edge.

\paragraph{Sharp thresholds and the unbalanced range.}
For fixed uniformity, our sufficient apex thresholds are exponential
in the number of base edges.  Determining how small the apex can be
while retaining the sharp exponent remains open even for complete
bipartite graphs.  Within the present method, a more immediate
question is whether the $3^{S+o_d(S)}$ threshold for ordered copies
can be attained when all embeddings must be excluded.  The gap
reflects the geometric independence required on every possible
apex side.  In the unbalanced problem, the central target remains
Conlon's linear range $\log_n m\le t/s$.
The polynomial interpolation degree and the degrees of the common
slices are the two costs that a substantially better construction
would have to reconcile.

\paragraph{Two forbidden colors in the off-diagonal problem.}
Theorem~\ref{thm:Ramsey} requires $K\ge3$.  It is natural to ask
whether, for $t$ sufficiently large in terms of $s$,
\[
\rrr_2(\mathcal K_{s,t};\mathcal K_n)
=\Theta_{s,t}\bigl((n/\log n)^s\bigr).
\]
The restriction is not merely a rounding loss in our proof.
The local-density estimate applies above a scale
$q^{(s+1)/2+s/(2h)}$, with $h\le s/2$; inserting this scale into
the present container estimate does not yield the independent-set
inequality needed for two relabelings.  Progress therefore calls
for stronger information about independent sets, or a different
construction, rather than only a refinement of the numerical
constants in the last calculation.
\section*{Acknowledgements}
The first author would like to thank Zichao Dong for bringing the random algebraic construction to his attention. As a participant in the ECOPRO Summer Student Research Program, the first author acknowledges the organizers, especially Hong Liu, and the members of IBS for their support and hospitality. The third author would like to thank Jie Ma and Boris Bukh for suggesting the use of symmetric polynomials in the construction. He also gratefully acknowledges that they originally posed and resolved the problem addressed in Corollary~1.12.
\appendix
\section{Spectral Estimate and Local Density}
\label{appendix:spectral}

This section is devoted to the proof of Lemma~\ref{lem:local-density}. To achieve the goal, we first construct a random symmetric matrix as follows. Let $k\ge s$ be positive integers and let $q$ be an odd prime power. Suppose that $\mathsf{W}$ is a pure $k$--dimensional variety of degree $D$ such that $|\mathsf{W}(\mathbb{F}_q)| \ge s$ and $\mathsf W(\mathbb F_q)$ is
$s$--wise $m$--independent. Denote $\rho \coloneqq k-s$.

Given positive integers $\delta_i \ge m$ where $i \in [\rho]$, we choose $(F, g_1,\dots, g_\rho) \in S^2 R_m \times R_{\delta_1} \times \cdots \times R_{\delta_\rho}$ uniformly and independently, and define the random symmetric matrix
$A = (A_{x,y})_{x,y\in\mathsf{W}(\mathbb{F}_q)}$ by
\begin{equation}\label{eq:random matrix}
A_{x,y} \coloneqq
\begin{cases}
\Big( 1_{F(x,y)=0}-\frac{1}{q} \Big)
\prod_{i=1}^{\rho}1_{g_i(x)=0}1_{g_i(y)=0} \quad &\text{~if~} x \ne y, \\
0 \quad &\text{~otherwise}.
\end{cases}
\end{equation}

Next, we establish an estimate for the spectral norm of $A$.

\begin{lemma}[Spectral Norm]\label{lem:spectral-norm}
For any positive integer $h\le s/2$ and sufficiently large prime power $q$, we have
\[
\Pr\!\Big[ \|A\|
\ge 16(h+1) D q^{\frac{s-1}{2}+\frac{s}{2h}} \Big] \le 2^{-4h}.
\]
\end{lemma}
\begin{proof}
Since $A$ is symmetric, $\|A\|^{2h}\le\operatorname{Tr}(A^{2h})$.
We will show that
\begin{equation}\label{lem:spectral-norm:eq1}
\mathbb E\operatorname{Tr}(A^{2h})
\le 2^{4h}(h+1)^{2h}D^{h+1}q^{sh+s-h}.
\end{equation}
Markov's inequality then gives the stated probability bound, since
$D\ge1$ and the $2h$-th power of the proposed norm threshold is
$2^{8h}(h+1)^{2h}D^{2h}q^{sh+s-h}$.

Expand the trace as a sum over closed walks of length $2h$ on
$\mathsf W(\mathbb F_q)$.  A walk with two consecutive equal
vertices contributes zero because the diagonal of $A$ is zero.
Every remaining walk visits at most $2h\le s$ vertices, so its
edge evaluations are mutually independent by
Lemma~\ref{edge-independence}.  Its contribution has expectation
zero if any unordered pair is traversed exactly once, because
that pair supplies one centered factor independent of all the others.
Thus only walks traversing each pair at least twice contribute.

For such a walk, let $a$ and $b$ be the numbers of vertices and
edges in its underlying simple graph.  This graph is connected,
so $2\le a\le b+1\le h+1$.  Each distinct pair contributes a
centered Bernoulli moment of absolute value at most $2/q$.
Independently, Lemma~\ref{g exp} shows that the probability that
all $a$ vertices survive the slicing equations is $q^{-a\rho}$.
The absolute expected contribution of the walk is therefore at most
$q^{-a\rho}(2/q)^b\le q^{-a\rho}(2/q)^{a-1}$.

There are at most $\binom{|\mathsf W(\mathbb F_q)|}{a}a^{2h}$
walks using exactly $a$ vertices, and
$|\mathsf W(\mathbb F_q)|\le2Dq^k$.
Since $k-\rho=s$, summing these bounds gives
\[
\mathbb E\operatorname{Tr}(A^{2h})
\le\sum_{a=2}^{h+1}\binom{2Dq^k}{a}a^{2h}
q^{-a\rho}(2/q)^{a-1}
\le 2^{4h}(h+1)^{2h}D^{h+1}q^{sh+s-h}.
\]
This proves \eqref{lem:spectral-norm:eq1}.
\end{proof}

Suppose that $s, \rho, Z, N$ are positive integers given in \eqref{subsec:pseudorandom-graph:parameters}. Let $\mathsf{W}$ and $\mathsf{X}$ be the varieties used in the construction of $\mathcal{G}$ in Lemma~\ref{probability of K_s,t-free}. Take $\delta_i = M_{\rho - i +1}(ks + 3)$ for each $i \in [\rho]$. By the definition of $\mathsf{W}$, $\mathsf{X}$ and $\mathcal{G}$, there is a randomly and uniformly chosen sequence $(F,g_1,\dots,g_{\rho}) \in S^2 R_m \times R_{\delta_1} \times \cdots \times R_{\delta_\rho}$, such that $\mathsf{X} = \mathsf{W} \cap \mathsf{V}(g_1,\dots, g_\rho)$, and 
\[
V(\mathcal{G}) = \mathsf{X}(\mathbb{F}_q),\qquad E(\mathcal{G}) = \{\{v_1,v_2\} \subseteq \mathsf{X}(\mathbb{F}_q): F(v_1,v_2) = 0\}.
\]

\begin{proof}[Proof of Lemma~\ref{lem:local-density}]
Let $A$ be the random matrix in~(20). Its principal submatrix
indexed by $X(\mathbb{F}_q)$ is
\[
    A_0=A_G-q^{-1}(J-I),
\]
where $A_G$ is the adjacency matrix of $G$, and $J$ and $I$
are the all-ones matrix and the identity matrix of the same order,
respectively. All entries of $A$ outside this block are zero,
so $\|A\|=\|A_0\|$.
Lemma~\ref{lem:spectral-norm} gives a single event of probability
at least $1-2^{-4h}$ on which
\begin{equation}\label{lem:local-density:eq1}
\|A_0\|\le L_h:=16(h+1)Dq^{(s-1)/2+s/(2h)}.
\end{equation}
On this event, take any $U\subseteq V(\mathcal G)$ of size
$u\ge R_h=8qL_h$.  With $\mathbf1_U$ its indicator vector,
\[
2|E_U(\mathcal G)|
=\mathbf1_U^{\tp}A_0\mathbf1_U+\frac{u^2-u}{q}
\ge\frac{u^2-u}{q}-L_hu
\ge\frac{u^2}{4q}
\]
for sufficiently large $q$.  The same spectral event gives this
inequality for every such $U$, which proves the stated simultaneous
probability bound.
\end{proof}

\bibliographystyle{abbrv}
\bibliography{ref}
\end{document}